\documentclass[11pt]{amsart}
\usepackage{latexsym}
\usepackage{amssymb}
\usepackage{amsmath}
\usepackage{amscd}
\usepackage{graphicx}
\usepackage{float}
\usepackage{mathrsfs} 
\usepackage{enumerate}
\usepackage{pgf,tikz}
\usetikzlibrary{arrows}
\usepackage{todonotes} 
\usepackage{hyperref}
\usepackage{tikz}
\usetikzlibrary{arrows,positioning,matrix,decorations.pathmorphing,calc,decorations.markings,decorations.pathreplacing,shapes} 
\usetikzlibrary{shapes.multipart}
\tikzset{node distance=1.5cm, auto}

\newtheorem{theorem}{Theorem}[section]
\newtheorem{question}[theorem]{Question}
\newtheorem{proposition}[theorem]{Proposition}
\newtheorem{lemma}[theorem]{Lemma}
\newtheorem{corollary}[theorem]{Corollary}

\theoremstyle{definition}
\newtheorem{definition}[theorem]{Definition}

\theoremstyle{remark}
\newtheorem{example}[theorem]{Example}
\newtheorem{remark}[theorem]{Remark}

\newcommand{\Vol}{\operatorname{Vol}}

\def\<{\langle}
\def\>{\rangle}
\def\R{\mathbb{R}}
\def\RR{\mathbb{R}}
\def\QQ{\mathbb{Q}}
\def\N{\mathbb{N}}
\def\CC{\mathbb{C}}
\def\C{\mathbb{C}}
\def\ZZ{\mathbb{Z}}
\def\Z{\mathbb{Z}}
\def\A{\mathcal{A}}

\def\e{\mathbf{e}}
\def\0{\mathbf{0}}

\def\F1{\mathcal{F}_1}

\def\mult{\operatorname{mult}}
\def\ker{\operatorname{ker}}
\def\Supp{\operatorname{Supp}}
\newcommand{\mb}{\mathbf}

\title{Sparse systems with high local multiplicity}
\date{}

\author{Fr\'ed\'eric Bihan}
\address{Laboratoire de Math\'ematiques\\
         Universit\'e Savoie Mont Blanc\\
         73376 Le Bourget-du-Lac Cedex\\
         France}

\email{frederic.bihan@univ-smb.fr}
\urladdr{http://www.lama.univ-savoie.fr/~bihan/}

\author{Alicia Dickenstein}
\address{Dto.\ de Matem\'atica, FCEN, Universidad de Buenos Aires, and IMAS (UBA-CONICET), Ciudad Universitaria, Pab.\ I, 
C1428EGA Buenos Aires, Argentina}
\email{alidick@dm.uba.ar}
\urladdr{http://mate.dm.uba.ar/~alidick}

\author{Jens Forsg{\aa}rd}
\email{jensforsgard@gmail.com }
 
\thanks{AD was partially supported by UBACYT 20020220200166BA and CONICET PIP 20110100580, Argentina. FB was partially supported 
by grant ANR-18-CE40-0009 ``ENUMGEOM'' of Agence Nationale de Recherche
and by AmSud project 20-MATH-02 (ARGO) }  

\subjclass{14B05,14Q99,52B20}
\begin{document}

\begin{abstract}
Consider a sparse system of $n$ Laurent polynomials in $n$ variables with complex coefficients and support in a finite lattice set $\A$. 
The maximal number of isolated roots of the system in the torus $(\mathbb C^*)^n$ is known to be the normalized volume of the convex hull of $\A$ 
(the BKK bound).   We explore the following question: if the cardinality of $\A$ equals $n+m+1$, what is the maximum  local intersection multiplicity
at one point in the torus  in terms of $n$ and $m$?  This study was initiated by Gabrielov~\cite{G} in the multivariate case.  We give an upper bound that is 
always sharp when $m=1$ and,  under a technical hypothesis, it is considerably smaller than the previous upper bound for any dimension $n$ and codimension $m$. We also present, for any value
of $n$ and $m$, a particular 
sparse system with high local multiplicity with exponents in the vertices of a cyclic polytope and we explain the rationale of our choice. 
Our work raises several interesting questions.
\end{abstract}

\maketitle

\section{Introduction}

Consider a univariate Laurent polynomial $f$ with support in a finite exponent set $\A \subset \Z$ and let $\mb q \in \C^*$.  It is easy to see that the multiplicity of $\mb q$ as a root of $f$
is strictly smaller than the cardinality $N$ of $\A$, independently of the degree $d$ of $f$. Up to multiplying $f$ by a monomial, we can assume that 
$\A \subset \{0, 1, \dots, d\}$ and that it contains $0$. The polynomial $(x-\mb q)^d$ has support in $\{0, 1, \dots, d\}$ with $d+1$ non-zero coefficients and it has multiplicity $d = (d+1)-1$ at $\mb q$. 
As long as $N < d+1$, the multiplicity of $\mb q$ as a root of $f$ is necessarily smaller than $d$. 
The upper bound $N -1$ was already given by G. H\'ajos~\cite{H} and a simple proof by induction on $d$ can be found in \cite[Lemma~3.9]{KS}.

In the multivariate case,  we fix a finite set $\A =\{a_0, \dots, a_{N-1}\} \subset \Z^n$ of cardinality $N$. 
We assume that the convex hull of $\A$ has full dimension $n$ and we  call 
\begin{equation}\label{eq:m}
m = N-(n+1)
\end{equation}
the codimension of $\A$.  To avoid trivial cases, we will assume that $N \ge n+2$, that is, that the codimension $m$ is at least $1$.
We will consider the  associated $(n+1) \times N$ integer matrix:
\begin{equation}\label{eq:A}
A = \left(\begin{array}{ccc} 1 & \dots & 1 \\ a_0 & \dots & a_{N-1} \end{array}
\right),
\end{equation}
with columns $(1, a_i) \in \Z^{n+1}$.  By our hypothesis about the point configuration $\A$, the matrix $A$ has maximal rank and thus we have that the codimension $m$  equals the dimension
of its kernel.

Consider a system of $n$ Laurent polynomials with complex coefficients supported on $\A$:
\begin{equation} \label{eq:fis}
f_i (x) = 
\sum_{j =0}^{N-1} c_{ij} \, x^{a_j} \, = \, 0, \quad i = 1, \dots, n.
\end{equation}
Here, $x = (x_1, \dots, x_n)$ and we will always assume that the coefficient matrix 
\[C=(c_{ij}) \in \C^{n \times N}\] 
has maximal rank $n$.  
We will also denote this system as
\begin{equation}\label{eq:fisC}
C \cdot x^A \, = \, 0, \end{equation}
where $x^{A}= (x^{a_0}, \dots, x^{a_{n+m}})^t$.
The local intersection multiplicity 
of an isolated solution $\mb q \in (\C^*)^n$ of a polynomial system $f_1, \dots, f_n$ (which is then a local complete intersection) is equal to the
dimension of the localization at $\mb q$ of the quotient of the ring of Laurent polynomials by the ideal they generate. It can also be defined via perturbation of the coefficients
or via topological (local) degree. This notion can be extended for instance to analytic germs. Our aim in this paper is to bound the local
multiplicity of an isolated solution in the torus in terms of $n$ and $m$.

The total number of isolated complex solutions in the torus of
the sparse system~\eqref{eq:fis} is at most the BKK bound (see \cite{Be,Kho, Kush}), which in this case is the normalized volume of the convex hull of $\A$.  The first upper bound 
$2^{ N-1 \choose 2} (n+1)^{N-1}$ on the number of 
real solutions of fewnomial systems is due to Khovanskii~\cite{Kh}. In turn, Gabrielov addressed in~\cite{G} the first  multivariate generalization of an upper bound 
for the multiplicity at an isolated solution $\mb q \in (\C^*)^n$  of a sparse polynomial system like~\eqref{eq:fis}.  He proved that the intersection multiplicity at $\mb q$
cannot exceed the similar bound $$2^{{N}\choose {2}}(n+1)^N,$$ independently of the degrees.
Our Theorem~\ref{T:t_{i}} and its Corollary~\ref{C:t_{i}bis} give, under some conditions, the considerably smaller bound $$(N-1)^{n}=(n+m)^n.$$
To prove this result, we use the combinatorial bound on the local multiplicity at the origin based on the notion of mixed covolume~\cite{HJS, E, M, KK}, that we recall
in Section~\ref{sec:tools}.  

A related interesting paper is~\cite{KS}, where Koiran and Skomra bound the intersection multiplicity at an isolated common zero 
of two bivariate  polynomials in terms of bounds for the degree of one of the polynomials and the number of terms of
the other one. We refer the reader to their Introduction, which also involves complexity issues.  
They asked  if this intersection multiplicity can be polynomially bounded in the number of monomials.  Our bound $(N-1)^n$  is polynomial in $N$ for any fixed $n$.
Another recent paper by Nikitin~\cite{N}  
addresses  the question of lower bounds for the maximal intersection multiplicity at a point in the torus, also  
in the particular case of two variables, allowing different supports for $f_1$ and $f_2$.  When both supports are equal, his lower bound for the
maximum local multiplicity equals $N-1$. 

\medskip

The intersection multiplicity at $\mb q=(q_1,\ldots,q_{n})$ of the system $f_1=\cdots=f_n=0$ equals the intersection multiplicity of the polynomials $f_i(x_1+q_1, \dots, x_n+q_n)$ at $\mb 0= (0, \dots, 0)$.
Moreover, 
a point $\mb q \in (\C^*)^n$ is a solution of system~\eqref{eq:fisC} if and only if $\mb 1 = (1,\ldots, 1)$ is a solution of the
 system with coefficient matrix $(c_{ij} \mb q^{a_j})$, obtained by right  multiplication of $C$ with an invertible diagonal
 matrix, with the same multiplicity.
 We will then assume in general in what follows that $\mb 1$ is an isolated solution of the system $C \cdot x^A=0$, or equivalently that $\mb 0$ is an isolated solution of
 the system $ C \cdot (x+ \mb 1)^A=0$ (and it has the same intersection multiplicity). This second system corresponds to the polynomials 
 \begin{equation}\label{eq:Fk}
 F_k(x) = f_k(x+{\mb 1}), \quad k= 1, \dots, n.
 \end{equation}

For any $n$ and $m$ we present the following explicit system given by the $n$-variate polynomials  
\begin{equation}
\label{eq:fij}
f_{m+1,n}(x)\, =\,  \cdots \, = \, f_{m+n,n}(x) \, = \, 0,
\end{equation}
defined
as follows for any $d \ge 0$:
\begin{equation}
\label{eqn:poly}
f_{d,n}(x) = \sum_{k=0}^{d}(-1)^k\, {d \choose k} \,x_1^k \,x_2^{k^2} \cdots x_n^{k^n}.
\end{equation}
All the supports are contained in the support of $f_{m+n,n}$, which has codimension
$m$. We prove in Proposition \ref{prop:finite} that ${\mathbf 1}=(1,\dots,1)$ is an isolated solution
of system~\eqref{eq:fij}. Then
Theorem~\ref{conj} shows that the intersection multiplicity of system~\eqref{eq:fij}
 at ${\mathbf 1}=(1,\dots,1)$ is at least
  \[{n+m \choose n} \, = \, {N-1 \choose n}.\]
 In fact, we conjecture that this multiplicity equals  ${n+m \choose n}$ and we prove this result
in case  $n=2$ for arbitrary $m$ in Proposition~\ref{prop:n=2}. 
Interestingly, our proof involves a result about the interlacing of roots of two particular hyperbolic Hermite polynomials.

We expect that, for any dimension $n$ and codimension $m$, the highest possible multiplicity  of any convenient (see Definition~\ref{def:various0}) Laurent polynomial system
with support in a fixed configuration $\A$ as in~\eqref{eq:fis}
at a  non-degenerate  (see Definition~\ref{def:various}) multiple solution in the torus, 
equals  ${n+m \choose n}$. 
We show in
Theorem~\ref{th:integer} that, under these conditions, the maximal multiplicity at the isolated common root $\mb 1$ can be
realized with polynomials with {\em integer coefficients}.  
We present in Section~\ref{sec:Aside} upper bounds for this multiplicity.

\medskip

We also consider the {\em associated Gale dual system} $g_1(y) =\cdots=g_m(y) =0$ defined in~\eqref{Gale system} at the origin in  $\C^m$, where $m \ge 1$
is the codimension of the system.
Such a Gale system is associated to a matrix $B$ and a reduced matrix $D$ which are Gale dual to $A$ and $C$ respectively.
The intersection multiplicity of system~\eqref{eq:fis}
at $\mb q \in (\C^*)^n$ equals the intersection multiplicity of system~\eqref{Gale system} at  $\mb 0 \in \C^m$ (see Theorem~\ref{Gale duality for one solution}. 
We introduce all the necessary notions in Sections~\ref{sec:tools} and~\ref{sec:Gale}.  
We also define a new notion of $H$-duality and study the local intersection multiplicity in this setting in Sections~\ref{sec:H} and~\ref{sec:ubound}.
 We summarize the relations among the different systems in Theorems~\ref{th:dual} and in Subsection~\ref{ssec:resume}.
Under conditions for the Gale system of having a convenient system and a non-degenerate solution at the origin,
 we get in Theorem~\ref{T:t_{i}*} and its Corollary~\ref{C:t_{i}bis*}  the bound $(N-1)^{N-1-n}= (n+m)^m$.
 
Note that  in the case $m=1$, we have  $(n+m)^m={n+m \choose n} = n+1$. In Section~\ref{sec:circuit} we prove that  
in this case our bound $n+1$ is sharp and we combinatorially characterize in Theorem~\ref{pseudo Vandermonde} 
those configurations $\A$ for which this bound
is attained. They correspond to the vertices of a lattice polytope affinely equivalent to a cyclic polytope with $n+2$ vertices.  Interestingly, the same condition is necessary
and sufficient for the existence of $n+1$ positive real roots, as we show in~\cite{BDF}.

\medskip

Two interesting questions related to the occurrence of positive solutions of real polynomial system that arise from our work are the following:

\begin{question}\label{realcomplex}
Is there a relation between the upper bound for the (global) number of positive real solutions of a real 
sparse polynomial systems and the (local) number of  intersection multiplicity at an isolated
solution in the {\em complex} torus of a complex sparse polynomial system with the same support?
\end{question}

 \begin{question}
\label{construction}
Does there exist a small perturbation of the coefficients in~\eqref{eqn:poly} giving a system with real coefficients and with (at least) $m+n \choose n $ positive solutions?
\end{question}
We answer positively Question~\ref{construction} in the codimension $1$ case in Section~\ref{sec:circuit}.

\medskip

Moreover, we pose the following question/conjecture about the symmetry between
dimension and codimension. A similar question concerning the number of {\em positive} roots has
been raised and partially studied in~\cite{BSS18}.

Note that if the original system and the corresponding Gale dual system are convenient and non-degenerate, then we
get from  Corollary~\ref{C:t_{i}bis} and Proposition~\ref{C:t_{i}bis*}
 the following symmetric upper bound for the intersection multiplicity at an isolated common root in the torus:
\[ (m+n)^{\min\{n,m\}}.\]

\begin{question}\label{q:Mnm}
 Call $M(n,m)$ the highest possible multiplicity at a multiple point in the torus of a Laurent polynomial 
system as in~\eqref{eq:fis} in $n$ variables and codimension $m$.  
Does it hold that
\begin{equation} \label{eq:M}
M(n,m) = M(m,n)? 
\end{equation}
\end{question}
We expect that the answer to Question~\ref{q:Mnm} is positive.

One obstacle we need to overcome
to prove affirmatively Question~\ref{q:Mnm} is summarized in the following
simple example. Let $\A = \{ 0,1,2,3\}$ and $f = (x-1)^3$. Then $1$ is
a non-degenerate multiple root of $f$ of multiplicity $3$ but  it is not true that for any
choice of matrices $B$ and $D$ reduced such that the associated Gale dual system is {\em convenient},
it holds that $(0,0)$ is a non-degenerate
 solution of the Gale dual system. This is similar to questions arising from the computation of
 the  Milnor number of an isolated singularity~\cite{K}.

  We then pose our next question. We refer the reader to Subsection~\ref{ssec:resume} for the notation.

  \begin{question}\label{q:ndc}
For any choice of matrices $A,C$ and Gale dual matrices $B,D$ as above, do there always exist matrices $\tilde{A}$ and $\tilde{C}$ left equivalent to $A$ and $C$ and matrices
$\tilde{B}$ and $\tilde{D}$ right equivalent to $B$ and $D$ such that the systems
$\tilde{C} \cdot (x+\mb 1)^{\tilde{A}}=0$ and the Gale system $(\tilde{D} \cdot y) ^{\tilde{B}}=1$
are convenient and non-degenerate at the origin?
\end{question}
  A positive answer to Question~\ref{q:ndc} would imply a positive answer to Question~\ref{q:Mnm}. Moreover, it would imply (as remarked before Question~\ref{q:Mnm})
   that for any solution of system~\eqref{eq:fis} the intersection
  multiplicity is bounded by $(m+n)^{{\rm min}(m,n)}$.

\medskip

We end this introduction by explaining the rationale behind our definition of the system  $f_{m+1}(x)= \dots = f_{n+m}(x) =0$ in Theorem~\ref{conj}. The first basic observation is the
following. We recall that the multiplicity of a point $\mb q$ of a hypersurface $f=0$ (or the multiplicity
of $f$ at $\mb q$) is defined as $0$ if $f(\mb q) \neq 0$ and otherwise,
as the maximal $k \ge 1$ such $f$ and all its partial derivatives up to order $k-1$ vanish at $\mb q$. 
Given $\A= \{a_0, \dots, a_{N-1}\} \subset \Z^n$ and  a polynomial $f =\sum_{j =0}^{N-1} c_{i} \, x^{a_j}$  with support $\A$,
it is straightforward to check that $f(x) = \frac{\partial}{\partial x_1}(f)(x)= \dots = \frac{\partial}{\partial x_n}(f) (x) =0$ at  ${\mathbf 1} = (1, \dots, 1)$ (that is, the multiplicity of
$f$ at $\mb 1$ is at least $2$)
if and only if the vector $c=(c_0, \dots, c_{N-1})$ lies in the kernel of the matrix $A$ defined in~\eqref{eq:A}
with columns $(1, a_i) \in \Z^{n+1}$.
A first idea to get a system of $n$ polynomials with support $\A$ and maximal intersection multiplicity at $\mathbf 1$ is to choose, if possible, all of them with
coefficients in ${\rm ker}(A)$ (i.e. the vector of coefficients lies in the $A$-discriminant~\cite{GKZ}). As we require that $C$ has maximal rank $n$, 
this is only possible if $n \le m$.
 On the other side, one could try to pick polynomials
with support $\A$ not only with a singularity at $\mathbf 1$ but also with the highest possible multiplicity at $\mathbf 1$. 
Following~\cite{DP}, we define the following matrices $ A^{(k)}$
associated to $\A$ and a natural number $k$.

\begin{definition}\label{def:Ak}
Let $\mb r_0 =(1, \ldots, 1), \mb r_1,
 \dots, \mb r_n \in \Z^N$ be the row vectors of the matrix $A \in \Z^{(n+1) \times (n+m+1)}$.
For any $\alpha \in \N^{n}$, 
denote by $\mb r_\alpha \in \Z^{N}$ the vector obtained as the evaluation
of the monomial $x^\alpha$ in the points of $\A$. For any positive integer $k$,  we define the associated matrix $A^{(k)}$ as follows. 
Order the vectors $\{\mb r_\alpha : |\alpha| \le k\}$ by degree and then by lexicographic
order with $0 < 1 < \dots < n$. Then, $A^{(k)}$ is the ${n+k \choose k} \times N$ 
integer matrix with these  rows.
As $\mb r_0 =(1, \ldots, 1)$,  the first $n+1$ row vectors of $A^{(k)}$ 
are just the row vectors $\mb r_0, \dots, \mb r_n$ of $A$.
\end{definition}

For example, if $n=1$ and $\mb r_1=(0,1,\ldots,m+1)$, then
\[{ A^{(k)}}=\left(
\begin{array}{cccccc} 1&1&1&1& \cdots &1\\
0&1&2&3& \cdots &(m+1)\\
0& 1& 4&9& \cdots & (m+1)^2\\
\vdots &\vdots &\vdots &\vdots & \cdots & \vdots\\
0&1&2^k&3^k&\cdots &(m+1)^k
\end{array}\right)
\] 
is a Vandermonde matrix.  

Again, it is straightforward to check for any support $\A$ that $\mb 1$ is a singular point of $f$ of
multiplicity $k \ge 1$ if and only if the vector $c$ of coefficients of $f$
lies in ${\rm ker}(A^{(k-1)})$ but not in ${\rm ker}(A^{(k)})$.
In particular $A^{(1)} = A$. 
When $k \ge 3$, its vector of coefficients $c$ lies in the singular
locus of the associated  $A$-discriminant.
These considerations allowed us to obtain a system with support $\A$ with high intersection
multiplicity at $\mb 1$.  Understanding the maximum possible $k$ is related to understanding
the singularities of the cuspidal components of the singular locus of the discriminant \cite{jens}.
There are interesting restrictions given by the geometry of the support set. For instance, if $\A \subset \Z^2$ 
with codimension $3$ (that is, $\# \A=6$ and $\A$ is not contained in a line), then $A^{(2)}$ is a $6 \times 6$
matrix. There exists a non-zero vector in ${\rm ker}(A^{(2)})$ (that is, there exists a polynomial $f$ with
support $\A$ and multiplicity at $\mb 1$ at least $3$) if and only if there exists a non-zero vector
in the left kernel of $A^{(2)}$. This condition means that all the points in $\A$ lie in a curve of degree $2$. 
We close our paper giving a sharp bound in Theorem~\ref{thm:MultiplicityOfOnePoint} in Section~\ref{sec:onef} 
for uniform configurations on the multiplicity at a point of a polynomial $f$ with support $\A$ only
in terms of $n$ and $m$, independently of the degree of the hypersurface $f=0$, and the general bound in Theorem~\ref{th:2} for planar configurations.

\medskip

\noindent {\bf Acknowledgement:} We thank Carlos D'Andrea for his help with the Laguerre polynomials in the proof of Proposition~\ref{prop:n=2}. We are also thankful to the
referee for the detailed comments that helped us correct and improve the exposition.

\medskip

\section{Multiplicities and mixed covolumes} \label{sec:tools}

 In this section, we recall the notion of mixed covolume and its relation with the notion of local multiplicity. 
The main references for this section are \cite{E, HJS, KK, M}.

  Let $\mathcal P$ be the set of all polyhedra $\Delta$ in $\R_{\geq 0}^{n}$ such that 
\[
B_{\Delta} := \R_{\geq 0}^{n} \setminus (\Delta + \R_{\geq 0}^{n})
\]
is bounded, where the sign plus stands for Minkowski sum. A polytope in
$\mathcal P$ is said to be  \emph{convenient}.
Equivalently, a convex polytope is {\em convenient} if and only if it intersects every coordinate axis. 
The \emph {mixed covolume} is the unique symmetric multilinear function $\Vol^{\circ}: {\mathcal P}^n \rightarrow \R_{\geq 0}$
 such that $\Vol^{\circ}(\Delta,\Delta,\ldots,\Delta)$ is the usual Euclidean volume of $B_{\Delta}$ multiplied by $n!$. Note that 
 if $\Delta_{1},\ldots,\Delta_{n} \in {\mathcal P}$ are polytopes with vertices in $\Z^{n}$, then $\Vol^{\circ}(\Delta_{1},\ldots,\Delta_{n})$ is an integer number.

Recall that the usual mixed volume $\Vol$ of $n$ convex sets in $\R^{n}$ is the unique symmetric multilinear function $\Vol$ such that 
$\Vol(\Delta,\Delta,\ldots,\Delta)$ is the usual Euclidean volume of $\Delta$ multiplied by $n!$, for any convex set $\Delta$. 
We have the following covolume formula for polytopes $\Delta_{1},\ldots,\Delta_{n} \in {\mathcal P}$, which is similar to a well-known formula for the usual mixed volume:

\begin{equation}
\label{alternating sum}
\Vol^{\circ}(\Delta_{1},\Delta_{2},\ldots,\Delta_{n})=
\sum_{I \subset [n], I \neq \emptyset} (-1)^{|I|}\Vol(B_{\Delta_I}),
\end{equation}
where $\Vol$ stands for the Euclidean volume and $\Delta_{I}=\sum_{i \in I} \Delta_{i}$.
Note that if $\Delta_{1},\ldots,\Delta_{n} \in \mathcal P$ then $\Delta_{I} \in \mathcal P$ for any $I \subset [n]$.

We now introduce some important definitions.

\begin{definition} \label{def:various0}
Consider a polynomial $F \in \C[x_{1},\ldots,x_{n}]$.
\begin{itemize}
\item The {\em Newton polytope} $\Delta(F)$ of $F$ is the convex-hull of the set $\Supp(F)$ of all the monomials occurring in $F$ with non-zero coefficient.
\item The {\em Newton diagram} ${\mathcal D}(F)$ is the union of the {\em bounded} faces of the convex-hull of the union of the sets $\alpha+\R_{\geq 0}^{n} $ for $\alpha \in \Supp(F)$. 
\item
A polynomial $F$ and its Newton diagram are called \emph{convenient} if $\Delta(F)$ is convenient.  A system of polynomials is {\em convenient} if each polynomial is {\em convenient}.
\end{itemize}
\end{definition}
Thus, a poynomial $F$ is convenient if and only its Newton diagram ${\mathcal D}(F)$ intersects each coordinate axis, equivalently, if $\Delta(F)$ contains a point in each coordinate axis.

\smallskip

\begin{remark}\label{rem:diagram} It follows from Definition~\ref{def:various0} that the Newton diagram ${\mathcal D}(F)$ of a polynomial $F \in \C[x_1, \dots, x_n]$ is 
the union of the bounded faces of the convex-hull of the union of the sets $\beta+\R_{\geq 0}^{n} $ for $\beta \in \Z_{\ge 0}^n$ such that
$\partial^\alpha(F)({\mb 0})=0$ for any $\alpha < \beta$ and $\partial^\beta(F)({\mb 0}) \neq 0$. Thus, if $\beta \in \Z_{\ge 0}^n$ lies in the Newton diagram
$\mathcal D(F)$ then $\partial^\alpha(F)({\mb 0})=0$ for any $\alpha < \beta$. Moreover, if $\beta$ is a vertex of $\mathcal D(F)$ then $\partial^\beta(F)({\mb 0}) \neq 0$.
\end{remark}

 \begin{example}\label{ex:diagram}
Let $F$ be a polynomial with the following Newton diagram, with four integer points $\alpha_1=(0,4), \alpha_2=(1,2), \alpha_3 = (3,1), \alpha_3=(5,0)$:
\begin{center}
\definecolor{xdxdff}{rgb}{0.490196078431,0.490196078431,1.}
\definecolor{ccqqqq}{rgb}{0.8,0.,0.}
\definecolor{qqqqff}{rgb}{0.,0.,1.}
\begin{tikzpicture}[line cap=round,line join=round,>=triangle 45,x=1.0cm,y=1.0cm]
\draw[->,color=black] (-0.5,0.) -- (5.5,0.);
\foreach \x in {,1.,2.,3.,4.,5.}
\draw[shift={(\x,0)},color=black] (0pt,2pt) -- (0pt,-2pt) node[below] {\footnotesize $\x$};
\draw[->,color=black] (0.,-0.5) -- (0.,4.5);
\foreach \y in {,1.,2.,3.,4.}
\draw[shift={(0,\y)},color=black] (2pt,0pt) -- (-2pt,0pt) node[left] {\footnotesize $\y$};
\draw[color=black] (0pt,-10pt) node[right] {\footnotesize $0$};
\clip(-0.5,-0.5) rectangle (5.5,4.5);
\draw [line width=2.8pt,color=ccqqqq] (0.,4.)-- (1.,2.);
\draw [line width=2.8pt,color=ccqqqq] (1.,2.)-- (5.,0.);
\begin{scriptsize}
\draw [fill=qqqqff] (0.,4.) circle (1.5pt);
\draw[color=qqqqff] (0.0758203516988,4.14792343124) node {$\alpha_1$};
\draw [fill=qqqqff] (1.,2.) circle (1.5pt);
\draw[color=qqqqff] (1.0740440815,2.15147597164) node {$\alpha_2$};
\draw [fill=xdxdff] (5.,0.) circle (1.5pt);
\draw[color=xdxdff] (5.0776725892,0.155028512032) node {$\alpha_4$};
\draw [fill=xdxdff] (3.,1.) circle (1.5pt);
\draw[color=xdxdff] (3.07049154111,1.15325224183) node {$\alpha_3$};
\draw [fill=qqqqff] (0.,3.) circle (1.5pt);
\draw [fill=qqqqff] (0.,2.) circle (1.5pt);
\draw [fill=qqqqff] (0.,1.) circle (1.5pt);
\draw [fill=qqqqff] (0.,0.) circle (1.5pt);
\draw [fill=qqqqff] (1.,0.) circle (1.5pt);
\draw [fill=qqqqff] (2.,0.) circle (1.5pt);
\draw [fill=qqqqff] (3.,0.) circle (1.5pt);
\draw [fill=qqqqff] (4.,0.) circle (1.5pt);
\draw [fill=qqqqff] (2.,1.) circle (1.5pt);
\draw [fill=qqqqff] (1.,1.) circle (1.5pt);
\end{scriptsize}
\end{tikzpicture}
\end{center}
This is equivalent to the following facts: $F(0,0) =0$ and all the partial derivatives $\partial^{(i,j)}(F)(0,0)=(0,0)$ for $(i,j)$ any other 
 integer point in the nonnegative orthant strictly below any of the points $\alpha_1, \alpha_2, \alpha_3, \alpha_4$, that is, for $(i,j)$ belonging to the set
 $$\{(0,0), (1,0), (2,0), (3,0), (4,0), (0,1), (1,1), (2,1), (0,2), (0,3)\},$$ 
 and moreover the derivatives $\partial^{\alpha_1}(F), \partial^{\alpha_2}(F),\partial^{\alpha_4}(F)$ corresponding to the vertices of $\mathcal D(F)$ are non-zero at the origin. There is no
restriction on the partial derivative $\partial^{\alpha_3}(F)(0,0)$, but this point imposes in particular the condition that $\partial^{(2,1)}(F) (0,0) = 0$. The polynomial $F$ can have other non-zero monomials
{\em above} its Newton diagram,
 but $B_{\Delta(F)}$ is the non-convex region between the positive coordinate axis and $\mathcal D(F)$.
 \end{example}

\begin{definition}\label{def:various}
Consider a polynomial system $G_{1}=\cdots=G_{n}=0$ in $n$ variables having an isolated solution at the origin.
For any vector $\epsilon \in \R_{>0}^{n}$, we denote by
$G_i^\epsilon$
 the subsum of terms in $G_i$ for which the inner product $\langle \epsilon, a_\ell\rangle$ is
minimum. The system is called \emph{non-degenerate at the origin} if $G_1, \dots, G_n$ are convenient and 
none of the initial systems $G_i^\epsilon=0, i=1, \dots, n$, have a common solution in the corresponding torus.  In general, given a system
$g_1 = \dots = g_n=0$ with an isolated solution at a point $\mb q=(q_1,\ldots,q_n)$, the system is called \emph{non-degenerate at $\mb q$} if the system
$G_i(x) = g_i (x_1+q_1, \dots, x_n+q_n)=0, i=1, \dots, n$, has a non-degenerate solution at the origin.
\end{definition}

The non-existence of a common solution  of the systems  $G_i^\epsilon=0, i=1, \dots, n,$ in Definition~\ref{def:various} can be ensured with the non-vanishing of appropriate
face resultants~\cite{GKZ}. 

\smallskip

We state two basic known results.

\begin{theorem}\cite[Theorem IX.8 and Proposition IX.29]{M}, \cite[Theorem 1]{E}, \cite{HJS, MBook}
\label{main co-mixed}
Let $F_{1},\ldots, F_{n} \in \C[x_{1},\ldots,x_{n}]$ be polynomials which vanish at the origin. 
If $F_{1},\ldots, F_{n}$ are convenient, then
\[
\mult_0(F_{1},\ldots, F_{n}) \geq  \Vol^{\circ}(\Delta(F_1),\Delta(F_2),\ldots,\Delta(F_n)).
\]
Moreover, if the polynomial system $F_{1}=\cdots=F_{n}=0$ is non-degenerate at the origin, then
$$\mult_0(F_{1},\ldots, F_{n})= \Vol^{\circ}(\Delta(F_1),\Delta(F_2),\ldots,\Delta(F_n)).$$
\end{theorem}

The definitions of Newton diagram, the origin being a non-degenerate isolated solution and the 
local multiplicity at the origin can be extended to power series (centered at the origin). Theorem~\ref{main co-mixed} 
also holds in this case.

{\begin{remark}\label{rem:diagrambis} Given convenient polynomials $F_{1},\ldots, F_{n} \in \C[x_{1},\ldots,x_{n}]$,
note that the mixed covolume of their Newton polytopes is determined by the corresponding Newton diagrams. So, we could write
\[ \Vol^{\circ}(\Delta(F_1),\Delta(F_2),\ldots,\Delta(F_n)) =  \Vol^{\circ}(\mathcal D(F_1), \dots, \mathcal D(F_n)).\]
\end{remark}

\begin{proposition}(Decreasing property)\cite[Example IX.3]{M}
 \label{decreasing property}
Consider  polytopes $P_{i},Q_{i} \in {\mathcal P}$, $i=1,\ldots,n$, such that $P_{i} \subset Q_{i}$ for all $i=1,\ldots,n$. Then
\[
\Vol^{\circ}(P_{1},P_{2},\ldots,P_{n}) \geq \Vol^{\circ}(Q_{1},Q_{2},\ldots,Q_{n}).
\]
\end{proposition}

We show in the following example that the property of having a system with an isolated non-degenerate solution in the torus is not necessarily preserved under an invertible monomial change of variables.

\begin{example} \label{counterexample1}\label{ex:counter}
Set  $\A=\{(0,0), (1,0), (0,1), (2,0), (0,2)\}$,  and consider the following polynomials with support in $\A$:
\begin{equation*}
f_1 = x^2 + y^2 -2x-y+1, \quad 
f_2=-x^2+2y^2+2x-y-2.
\end{equation*} 
They have an isolated singular solution at ${\bf 1} = (1,1)$. Then, the translated polynomials
$F_1  = f_1(x+1, y+1) = x^2 + y^2 +y$, $F_2 = f_2(x+1, y+1) =-x^2 + 2y^2 + 3y$ have an isolated singular
solution at ${\bf 0} = (0,0)$. 

The translated system $F_1=F_2=0$  is convenient and
{\bf non-degenerate} at the origin. The multiplicity at the origin equals $2$, which is the
mixed covolume of twice the segment with vertices $\{(2,0), (0,1)\}$ (this segment is the common Newton diagram of $F_1$ and $F_2$).

We now perform the monomial change of coordinates $x=u^2v, y=u^5v^3$, with determinant of
the exponents equal to $1$, sending $\bf 1$ to $\bf 1$. If we set $g_i = f_i (u^2v,u^5v^3)$, $i=1,2$, we get
the polynomials
\begin{equation*}
g_1  =  u^{10}v^6-u^5v^3+u^4v^2-2u^2v+1, \\
g_2  =   2u^{10}v^6-u^5v^3-u^4v^2+2u^2v-2.
\end{equation*}
The multiplicity at $\bf 1$ of $g_1, g_2$ is also equal to $2$.
Now, we compute $G_i = g_i (u+1, v+1)$, $i=1,2$:
\begin{equation*}   G_1 = 5u+3 v + \dots,  \quad
G_2 =15 u + 9v + \dots,\end{equation*}
where the dots indicate higher order summands.
So, as $15 u + 9v = 3 (5u+3v)$, this system is convenient but {\bf degenerate} a the origin.
\end{example}

\section{Bounds for the local intersection multiplicity}\label{sec:Aside}

In this section we prove in Theorem~\ref{th:integer} that the maximal local multiplicity of an isolated solution (under generic hypotheses) can always be attained by
polynomials with integer coefficients. We also give in Theorem~\ref{T:t_{i}} an upper bound for this multiplicity only in terms of the dimension and codimension
of the system. Some interesting consequences are proved in Corollaries~\ref{C:t_{i}bis}, ~\ref{C:t_{i}} and~\ref{C:d}, and in  Proposition~\ref{P:bettern=2}.

We will use the setting and the notations in the Introduction.
Consider a system $f_1=\dots=f_n=0$
as in~\eqref{eq:fis}
of $n$ polynomial equations and $n$ variables, with exponents in a set
\[
\A=\{a_{0},a_{1},\ldots,a_{n+m}\} \subset \ZZ^{n}
\]
consisting of $n+m+1$ points and {\em not contained in an affine hyperplane}.
As before, we also denote this system as
$C \cdot x^A \, = \, 0$, 
where $C = (c_{ij}) \in \C^{n \times (n+m+1)}$ is the coefficient matrix of the system of rank $n$ and $A\in \Z^{(n+1) \times (n+m+1)}$ is the matrix
with columns $(1,a_0), \dots, (1, a_{n+m})$. Recall that
our assumption that the convex hull of the exponents $a_{i}, i=0, \dots, n+m$, has full dimension $n$ means that this matrix $A$ has maximal rank $n+1$.
As we mentioned, our aim is to bound the local
multiplicity of an isolated solution $\mb q \in (\C^*)^n$ of the system. As we pointed out, it is enough to consider the case $\mb q= \mb 1$, or equivalently
to assume that $\mb 0$ is an isolated solution of
 the system $ C \cdot (x+ \mb 1)^A=0$, that we also denote in~\eqref{eq:Fk} by $F_k(x) = f_k(x+{\mb 1}), \quad k= 1, \dots, n$.
 Moreover, by adding an integer vector $a$ to all points in $\A$ (or equivalently, by multiplying the polynomials with support $\A$ with the monomial $z^a$), we can assume without loss of generality that $\A$
  lies in the non-negative orthant $\Z_{\geq 0}^n$. Left multiplication of $A$ by an invertible matrix amounts to do a monomial change of coordinates for the complex torus.  Note this action preserves the
  multiplicity of the system at $\bf 1$ (here the coefficient matrix is fixed). Similarly, left multiplication of the coefficient matrix $C$ by an invertible matrix does not change
  multiplicity of the system at $\bf 1$. We will see in Propositions \ref{prop:conv} and \ref{prop:convenient} that
  we can always use any of these actions to make the system $ C \cdot (x+ \mb 1)^A=0$ convenient.
  
\subsection{Systems with integer coefficients}\label{ssec:integer}

We next  prove  that the maximal multiplicity at the non-degenerate common solution $\mb 1$ 
of a convenient sparse system can be achieved with a system with integer coefficients.
 
\begin{theorem}\label{th:integer}
Let $C \in \C^{n\times (n+m+1)}$ and $A \in \Z_{\geq 0}^{(n+1) \times(n+m+1)}$ such that the system $C \cdot (x+\mb 1)^A=0$ is convenient and non-degenerate at $\mb 0$ with local intersection multiplicity $\mu$. 
Then there exists an {\em integer} matrix
$C^0 \in \Z^{n \times (n+m+1)}$ such that the system $C^0 \cdot (x+{\mb 1})^A=0$ is convenient and non-degenerate at
$\mb 0$ of the same multiplicity $\mu$.
\end{theorem}

\begin{proof}
Let $C \cdot (x+\mb 1)^A=0$ and consider the polynomials $f_1, \dots, f_n$ with
support in $\A$ associated with the rows $c_1, \dots, c_n$ of $C$.  
Denote  $F_k(z) = f_k(z+ \mb 1)$ as in~\eqref{eq:Fk}, $k=1, \dots, n$.  
Assume that these polynomials define a convenient and non-degenerate system with an isolated intersection point at $\mb 0$ with local intersection multiplicity $\mu$.
Then, $\mu = \Vol^{\circ}(\mathcal D(F_1), \dots, \mathcal D(F_n))$ by Theorem~\ref{main co-mixed}.
We will show that we can find {\em integer} coefficient vectors $(c^0_1,\dots, c^0_n)$ with the same Newton diagrams respectively (and are thus convenient) and which moreover define a non-degenerate system at $\mb 0$. The matrix $C^0$ with
these rows will then satisfy the statement of the theorem.

We have that $F_k, k=1, \dots, n$, is of the form
\[ F_k(z) = \sum_{j=0}^{m+n} c_{kj} \, (z+ \mb 1)^{a_j} =  \sum_{j=0}^{m+n} c_{kj} \, \prod_{i=1}^n \sum_{\ell_i =0}^{a_{ij}} {a_{i j} \choose {\ell_i}} z_i^{\ell_i},\]
and so its monomial expansion equals
\begin{equation}
\label{eq:F}
F_k(z) = \sum_{\alpha} \left( \sum_{j=0}^{n+m} c_{kj}  \prod_{0 \le \alpha_i \le a_{ij}} {a_{ij} \choose {\alpha_i}} \right)  \, z^\alpha = \sum_{\alpha} \frac{\partial^\alpha(F_k) (\mb 0)}{\alpha_1! \dots \alpha_n!} \, z^\alpha.\end{equation}
The sums are over the exponents $\alpha \in \Z_{\ge 0}^n$ such that for any $i=1, \dots, n$,  we have the inequality $\alpha_i \le  {\rm max}\{a_{ij}, j=0,\dots, n+m\}$. 
Note that each derivative $\partial^\alpha(F_k) (\mb 0)$  is a homogeneous
linear form in the vector of coefficients of the polynomial $f_k$.

Let $\{\gamma^1, \dots, \gamma^s\}$ be the integer points in $\mathcal D(F_k)$
 Note that all $\gamma^i$ are non-zero because $\mb 0$ is a solution.
 Following Remark~\ref{rem:diagram} and Example~\ref{ex:diagram}, we have that
 the vector of coefficients $c_k =(c_{kj})$ lies in the kernel of the integer matrix $A'_k$ 
with rows  indicated by
subindices  in the (finite) set $\Sigma_k = \cup_{j=0}^{s} \{ 0 \le \alpha < \gamma^j\}$. For $\alpha \in \Sigma_k$, the $\alpha$th row of $A'_k$ is 
 the  linear form $\partial^\alpha(F_k) (\mb 0)$ in the expansion~\eqref{eq:F}. Moreover, $\partial^\alpha(F_k) (\mb 0) \neq 0$ for any vertex of $\mathcal D(F_k)$ and since $F_k$ is
 convenient, there exist vertices in each coordinate axis.

Let $\mathcal L$ be the linear subspace
\begin{equation*}
\mathcal L \, = \, \{c'= (c'_1, \dots, c'_n) \in \C^{n(n+m+1)} \, : \, c'_k \in {\rm ker}(A'_k), k= 1,\dots, n\}.
\end{equation*}

As $c_k$ is a non-zero vector and $A'_k$ is an integer matrix, we can find non-zero vectors in its kernel with {\em integer} coefficients, for any $k=1, \dots, n$. Thus,  $\mathcal L$ has a basis of integer vectors.
We need to find vectors of coefficients that are not only integer and lying in $\mathcal L$, but they should also satisfy several nonvanishing conditions.

On one side, consider the polynomial
\[P_k = \prod_{\alpha  \text{ a vertex of }  \mathcal D(F_k)} \partial^\alpha(F_k) (\mb 0) .\]
Then, $P_k(c_k) \neq 0$ (where $P_k$ is seen as a homogeneous
linear form in the vector of coefficients of the polynomial $f_k$).
 Denote by $P= \prod_{k=1}^n P_k$.  Any polynomial system with coefficients  $c'=(c'_1, \dots, c'_n) \in L$  satisfying $P(c') \neq 0$ corresponds to polynomials with the same Newton diagrams
 as $F_1, \dots, F_n$.

 Moreover, the system $F_1, \dots, F_k$ being non-degenerate is expressed by the fact that  a finite set of {\em face resultants}~\cite{CDS} 
 $\{R_{j}, j \in \mathcal F\}$ corresponding to certain faces of the Newton diagrams, are non-zero on the coefficients $c=(c_1, \dots, c_k)$ of the system. 
 These face resultants are homogeneous polynomials. They are not identically zero  because they do not vanish on $c$. Denote by $R_{\mathcal F} (c')= \prod_{j \in \mathcal F} R_{j}(c'_1, \dots, c'_k)$.

Let $Q$ be the homogeneous polynomial in the coefficients $c'=(c'_1, \dots, c'_n)$ defined as
the product
\[ Q \, = \, P  \, \cdot \,  R_{\mathcal F}.\]
 Let $V \subset L$ be the open cone in $\mathcal L$ where $Q \neq 0$.  Then $V$ is non-empty because we are assuming that $c=(c_1,\dots, c_n) \in V$, thus $V$ contains a rational point and a fortiori, also a point with
 integer coordinates. So
 we can choose integer vectors of coefficients $(c^0_1,\dots,c^0_n) \in V$. Their associated polynomials have the same Newton diagrams $\mathcal D(F_1),
\dots, \mathcal D(F_n)$ respectively (and are in particular convenient). Moreover, they define a non-degenerate system at $\bf 0$, which implies that their intersection multiplicity at the origin equals $\mu$.
\end{proof}

We have the following immediate corollary of Theorem~\ref{th:integer}.

\begin{corollary} \label{cor:integer}
 Given a finite lattice configuration $\A \subset \Z^n$ with full dimensional convex hull, the maximal possible multiplicity of an isolated 
singularity at $\mb 1$ of a  system $f_1 =\dots = f_n=0$  with exponents in $\A$  as in~\eqref{eq:fij} for which the system $f_1(x+\mb 1)= \dots = f_n(x+ \mb 1) =0$
 is convenient and non-degenerate at the origin,   is attained by polynomials with integer coefficients. 
\end{corollary}

\subsection{Explicit bounds} \label{ssec:explicit}

We next prove our main Theorem~\ref{T:t_{i}} with an upper bound on the local multiplicity in terms of $n$ and $m$.
We will assume without loss of generality that $\A$ lies in the non-negative orthant.

As usual, we will denote by $e_1, \dots, e_n$ the canonical basis in $\R^n$. 
We will bound the coordinate of the intersection point with a coordinate axis of the Newton diagram of any polynomial $F_k=f_k(x+ \mb 1)$ in
a convenient system of polynomials $F_1, \dots, F_n$ with $f_1, \dots, f_k$ having support $\A$ and an isolated intersection point at $\mb 1$.
More precisely, given $k$ and $\ell$, we give  in Lemma~\ref{lem:axis} below, an upper bound for the first derivative $\partial^{\beta_{k,l} e_\ell}$ of $F_k$ with respect to $x_\ell$ which is non-zero at $\mb 0$ 
in terms of  $\A$. This will allow us to get an upper bound for the mixed covolume $\Vol^{\circ}(\Gamma_{1},\ldots,\Gamma_{n})$ of the Newton diagrams of $F_1, \dots, F_n$.
Under the additional assumption that the system of polynomials $F_1, \dots, F_n$ is non-degenerate at $\bf 0$, this will provide (using Theorem \ref{main co-mixed}) an upper bound for the local multiplicity at the origin
of this system, see Theorem \ref{T:t_{i}}.

We first simplify the description of the linear forms in the coefficients in~\eqref{eq:F}.  It is more convenient to
work with the toric (a.k.a. Euler) derivatives $\theta_i = x_i \partial^{e_i}$ instead of the standard partial derivatives. We write $\theta^\alpha = \theta_1^{\alpha_1} \dots \theta_n^{\alpha_n}$.
We use the partial order on $ {\ZZ}_{\geq 0}^{n}$
given by $\alpha \leq \beta$ if $\alpha_{i} \leq \beta_i$ for $i=1,\ldots,n$, and we write $\alpha < \beta$
when $\alpha \leq \beta$  and $\alpha \neq \beta$. For instance, we have $(1,2) < (2,2)$. 

 We introduce a useful notation.

\begin{definition} \label{def:LA}
Given matrices $C, A$  as before, for any $\alpha \in \Z_{\geq 0}^n$ we define
\begin{equation}\label{L_{f}}
L_{{k}}(\alpha) = \sum_{j=0}^{m+n}c_{kj} a_j^\alpha.
\end{equation}
We have $L_k(\alpha)= \theta^\alpha(f_k)(\mb 1)$, where as before $f_k= \sum_{j=0}^{n+m} c_{kj} \, x^{a_j}$.
\end{definition}

The following lemma is a straightforward consequence of Remark~\ref{rem:diagram} and the Leibniz rule:

 \begin{lemma}
\label{lem:tor}
Let $f \in \C[x_1, \dots, x_n]$  and $\mb q \in \C^n$ and
set $F$ in $\C[x_1, \dots, x_n]$ as $F(x) = f(x_1 + q_1, \dots, x_n + q_n) = f(x+\mb q)$.
Assume  $\beta \in \Z_{\ge 0}^n$. 

\begin{itemize}
\item
We have $\partial^\alpha(F)({\mb 0})=0$ for any $\alpha \in \Z_{\ge 0}^n$ such that $\alpha < \beta$
 if and only $\theta^\alpha(f)(\mb q)=0$ for any $\alpha \in \Z_{\ge 0}^n$ such that $\alpha < \beta$. In this case
\begin{equation*} \label{eq:theta}
\theta^\beta(f)(\mb q) \, = \, \mb q^\beta \, \partial^\beta(F)(\mb 0).
\end{equation*}
In particular, if $\mb q= {\bf 1}$ we have the equality $\theta^\beta(f)({\bf 1}) \, = \partial^\beta(F)(\mb 0)$.

\item  Assume $\mb q= {\bf 1}$.
If $\beta$ lies in ${\mathcal D}(F_{k})$ then $L_{{k}}(\alpha) =0$ for any $\alpha \in \Z_{\geq 0}^{n}$ such that $\alpha < \beta$ and
the coefficient of $x^{\beta}$ in the expansion~\eqref{eq:F}  of $F_{k}(z)$ is equal to
$\frac{L_{k}(\beta)}{\beta_{1}! \cdots \beta_{n}!}$. 
\end{itemize}
\end{lemma}
 
\smallskip

To simplify the notation, we will assume from now on, as before, that $\mb q= \mb 1$.
We now define three quantities that occur in the bounds of Theorem~\ref{T:t_{i}}.

\begin{definition} \label{def:rhost} 
Let
 $k, \ell \in \{1,\ldots,n\}$. For $u \in \{a_{\ell 0},\ldots, a_{\ell \,  m+n}\}$,
define
\begin{equation}
\label{E:sum}
\bar{c}_{k,\ell,u}=\sum_{a_{\ell j}=u} c_{k j}.
\end{equation}
We also set
\begin{equation}\label{E:cardCA}
\rho_{k,\ell} =\left| \{u \in  \{a_{\ell 0},\ldots, a_{\ell \,  m+n}\} \, , \bar{c}_{k,\ell,u} \neq 0\}\right|-1,
\end{equation}
\begin{equation}\label{E:cardC}
s_k = \left| \{  j \in \{0,\ldots, m+n\} \, , c_{k j} \neq 0\}\right|-1,
\end{equation}
and
\begin{equation}\label{E:card}
t_\ell = \left| \{ a_{\ell j}, \, j=0,\ldots, m+n\}\right|-1.
\end{equation}
\end{definition}

Note that $\rho_{k,\ell} \leq \mbox{min} (s_{k},t_{\ell})$.
Recall that the Newton diagram $\mathcal D(F_{k})$ is convenient when it intersects
each coordinate axis. Equivalently, $\mathcal D(F_{k})$ is convenient when $F_{k}$ is not identically zero
on any coordinate axis.

\begin{lemma}
\label{lem:axis}
Let $k,\ell=1,\ldots,n$.
The polynomial $F_{k}=f_k(x+\mb 1)$ is identically zero on the $\ell$-coordinate axis if and only if the sum
$\bar{c}_{k,\ell,u}$ defined in~\eqref{E:sum} vanishes for all $u \in  \{a_{\ell 0},\ldots, a_{\ell m+n}\}$, equivalently, when $\rho_{k,\ell}=-1$.

Assume that $F_{k}$ is not identically zero on the $\ell$-coordinate axis and let 
$\beta=\beta_{k,\ell} \cdot \e_{\ell}$ be the intersection point of $\mathcal D(F_{k})$ with the $\ell$-coordinate axis.
Then $\beta_{k,\ell} \le \rho_{k,\ell}\ \leq \mbox{min} (s_{k},t_\ell)$, where $\rho_{k,\ell}$, $s_{k}$ and $t_{\ell}$ are as in Definition~\ref{def:rhost}. 
\end{lemma}

\begin{proof}
Using Lemma~\ref{lem:tor} we get that $F_{k}$ is identically zero on the $\ell$-coordinate axis if and only if
$L_{k}(\lambda \cdot \e_{\ell})=0$ for any $\lambda \in \Z_{\geq 0}$, which is equivalent to 
\begin{equation}
\label{eq:vandeq}
\sum_{j=0}^{m+n} c_{kj} a_{\ell j}^{\lambda}=0, \; \mbox{for any $\lambda \in \Z_{\geq 0}$}.
\end{equation}
Assume for simplicity that $a_{\ell,0},\ldots,a_{\ell, t_\ell}$ are distinct so that
$\{a_{\ell 0}, a_{\ell 1},\ldots,a_{\ell, m+n}\}=\{a_{\ell 0}, a_{\ell 1},\ldots,a_{\ell, t_{\ell}}\}$.
Write $\bar{c}_{s}$ for $\bar{c}_{k,\ell, a_{\ell s}}$, $s=0,\ldots,t_{\ell}$.
For any positive integer $p$ consider the Vandermonde matrix
\begin{equation}
\label{E:Vandermonde}
V_{p}=\big(a_{\ell s}^{\lambda}\big)_{0 \leq \lambda \leq p, 0 \leq s \leq t_{\ell}}.
\end{equation}
Since $A$ has maximal rank, at least one coefficient  $a_{\ell s}$ is non-zero.
Therefore~\eqref{eq:vandeq} is satisfied if and only if the column vector $(\bar{c}_{s})_{s=0,\ldots,t_{\ell}}$
belongs to the kernel of $V_{p}$ for any positive integer $p$, equivalently, $(\bar{c}_{s})_{s=0,\ldots,t_{\ell}}$ is the zero vector.

Assume now that $\beta = \beta_{k,\ell} \cdot \e_\ell \in \mathcal D(F_{k})$. Then $\beta$ is a vertex of $\mathcal D(F_{k})$. Thus, by Lemma~\ref{lem:tor} we have
$L_{k}(\beta) \neq 0$ and $L_{k}(\lambda \cdot \e_{\ell})=0$ for any non-negative integer $\lambda < \beta_{k,\ell}$.
Let $\tilde{c}$ be the column vector obtained by deleting the zero coordinates of $(\bar{c}_{s})_{s=0,\ldots,t_{\ell}}$.
Then  $\tilde{c}$ is a non-zero vector which belongs to the kernel
of the Vandermonde matrix $\tilde{V}_{\beta_{k,\ell}-1}$  obtained by removing the corresponding columns of $V_{\beta_{k,\ell}-1}$. The number of columns of $\tilde{V}_{\beta_{k,\ell}-1}$ equals 
$\rho_{\ell,k}+1$ and its number of rows equals $\beta_{k,\ell}$. Since $a_{\ell 0},a_{\ell 1}, \ldots, a_{\ell, t_{\ell}}$ are distinct, this implies the inequality $\beta_{k,\ell} \leq \rho_{\ell,k}+1-1=\rho_{\ell,k}$.
\end{proof}

\begin{proposition}
\label{prop:conv}
Assume that $\bf 1$ is an isolated solution of the system $C \cdot x^{A}=0$. Then, there exists an invertible matrix $M \in \C^{n \times n}$ such that
the equivalent system $M\, C.(x+\mb1)^A=0$ is convenient.
\end{proposition}

\begin{proof}
Let $\ell \in \{1,\ldots,n\}$.
By Lemma ~\ref{lem:axis}, we have $\bar{c}_{k,\ell,u}=0$ for all $k=1,\ldots,n $ and all $u \in  \{a_{\ell 0},\ldots, a_{\ell m+n}\}$ if and only if
the $\ell$-coordinate axis is contained in the set of solutions of $F_{1}=\cdots=F_{n}=0$, which would imply that $\bf 1$ is not an isolated solution of the system $C \cdot x^{A}=0$.
So assuming that $\bf 1$ is an isolated solution of the system $C \cdot x^{A}=0$, we obtain the existence of $k$ and $u$ such that $\bar{c}_{k,\ell,u} \neq 0$.

Now, for any power series $F_1, \dots, F_n$, left multiplying their coefficient matrix by a generic invertible matrix $M$ produces  $n$ power series  with support equal to
the union of the supports of $F_1, \dots, F_n$, and thus for any generic invertible matrix $M$,  the system
$M\, C \cdot (x+ \mb 1)^A=0$ is convenient.
\end{proof}

\begin{proposition}
\label{prop:convenient}
If $a_{i,0},\ldots,a_{i,n+m}$ are distinct for $i=1,\ldots,n$ then the system $C.(x+\mb 1)^{A}=0$ is convenient.
In general, there exists an invertible matrix $M' \in \Z^{(n+1) \times (n+1)}$ such that
the system $C.(x+\mb1)^{M'A}=0$ is convenient.
\end{proposition}

\begin{proof}
Assume first that $a_{i,0},\ldots,a_{i,n+m}$ are distinct for $i=1,\ldots,n$ . Then by Lemma ~\ref{lem:axis}. the polynomial $F_{k}$ is identically zero on the $\ell$-coordinate axis if and only if
$c_{k,0}=\cdots=c_{k,n+m}=0$. But no row vector of $C$ is zero since $C$ has maximal rank and we conclude that $C \cdot (x+\rm 1)^{A}=0$ is convenient.
The general case follows from the previous one using the fact that since $A$ has maximal rank there exists an invertible matrix $M' \in \Z^{(n+1) \times (n+1)}$ such that $M'A$ has a first row of ones
and distinct coefficients in each other row.
\end{proof}

Consider the polytopes
\begin{equation}
\label{eq:Dpolytope}
\Gamma_{k}=\mbox{conv} \, \{\rho_{k,\ell} \cdot \e_{\ell}, \, \ell=1,\ldots,n\}, \quad k=1,\ldots,n,
\end{equation}
where $\rho_{k,\ell}$ is defined in~\eqref{E:cardCA}.

\begin{theorem}
\label{T:t_{i}}
Let $C \in \C^{n\times (n+m+1)}$ and $A \in \Z_{\geq 0}^{(n+1) \times(n+m+1)}$ such that the system $C \cdot (x+\mb 1)^A=0$ is convenient and non-degenerate at $\mb 0$ with local intersection multiplicity $\mu$.
Then
$$
\mu \leq  \Vol^{\circ}(\Gamma_{1},\ldots,\Gamma_{n}) \leq \mbox{min}\left(\prod_{i=1}^{n} s_i,\prod_{i=1}^{n} t_i \right)
$$
where $s_1, \ldots, s_{n}$ and $t_1,\ldots, t_{n}$  are defined in~\eqref{E:cardC} and~\eqref{E:card}.
\end{theorem}
\begin{proof}
By Theorem~\ref{main co-mixed} we have $\mu=\Vol^{\circ}(\Delta(F_{1}),\ldots,\Delta(F_{n}))$, where $\Delta(F_{i})$ stands for the Newton polytope of $F_{i}$.
Denote by $T$ the convex hull of $\{t_{\ell} \cdot \e_{\ell}, \, \ell=1,\ldots,n\}$. For $k=1,\ldots,n$,
denote by $S_{k}$ the convex hull of $\{s_{k} \cdot \e_{\ell}, \, \ell=1,\ldots,n\}$. Lemma ~\ref{lem:axis} yields the following inclusions for $k=1,\ldots,n$:

\begin{equation}
\label{eq:T}
T +\R_{\geq 0}^{n} \subset  \Gamma_{k}+\R_{\geq 0}^{n} \subset \Delta(F_{k})+\R_{\geq 0}^{n}
\end{equation}
and
\begin{equation}
\label{eq:S}
S_{k} +\R_{\geq 0}^{n} \subset  \Gamma_{k}+\R_{\geq 0}^{n} \subset \Delta(F_{k})+\R_{\geq 0}^{n}.
\end{equation}
Recall that the mixed covolume is decreasing (see Proposition~\ref{decreasing property}). Hence using~\eqref{eq:T} we get 
$\mu=\Vol^{\circ}(\Delta(F_{1}),\ldots,\Delta(F_{n})) \leq \Vol^{\circ}(\Gamma_{1},\ldots,\Gamma_{n})  \leq \Vol^{\circ}(T,\ldots,T)$.
By definition of the mixed covolume, we have
$\Vol^{\circ}(T,\ldots,T)= n! \cdot  \Vol(B_{T}) = t_1 \cdots t_{n}$
and we conclude that $\mu \leq \Vol^{\circ}(\Gamma_{1},\ldots,\Gamma_{n})  \leq t_1 \cdots t_{n}$.

Similarly,~\eqref{eq:S} implies $$\mu=\Vol^{\circ}(\Delta(F_{1}),\ldots,\Delta(F_{n})) \leq \Vol^{\circ}(\Gamma_{1},\ldots,\Gamma_{n})  \leq \Vol^{\circ}(S_{1},\ldots,S_{n}).$$
By multilinearity of the mixed covolume we have $\Vol^{\circ}(S_{1},\ldots,S_{n})=(\prod_{k=1}^{n} s_{k})\cdot \Vol^{\circ}(\Delta,\ldots,\Delta)=(\prod_{k=1}^{n} s_{k})\cdot n! \cdot  \Vol(B_{\Delta})$ 
where $\Delta$ is the convex hull of the points $\e_{\ell}$ for $\ell=1,\ldots,n$, and we conclude that $\mu \leq \Vol^{\circ}(\Gamma_{1},\ldots,\Gamma_{n})  \leq s_1 \cdots s_{n}$.
\end{proof}

\begin{remark}\label{rm:Haas}
Note that $s_{k}+1$ is the cardinality of the support of the $k$-th equation of the polynomial  system $C \cdot x^{A}=0$.
Interestingly enough the bound $s_1 \cdots s_{n}$ coincides with the Kouchnirenko bound which was conjectured to be a sharp bound for the number of positive solutions
of $C \cdot x^{A}=0$ assuming that $C$ has real coefficients. This global bound was disproved, for example in~\cite{Haas}. 
Note that by Theorem~\ref{T:t_{i}}, we prove that the Kouchnirenko bound is instead an upper bound for the local multiplicity at a {\em non-degenerate} 
multiple solution in the complex torus of a polynomial system with complex coefficients (compare with Question~\ref{realcomplex}).
\end{remark}

\subsection{Consequences and improvements of Theorem~\ref{T:t_{i}}} \label{ssec:conseq}
We deduce the following consequences and improvements of Theorem~\ref{T:t_{i}}: 
Corollaries~\ref{C:t_{i}bis} and~\ref{C:t_{i}}, Proposition~\ref{P:bettern=2} in the case $n=2$ and Corollary~\ref{C:d} in terms of a bound
for all exponents. We immediately get the following general upper bound.

\begin{corollary}
\label{C:t_{i}bis}
Let $C \in \C^{n\times (n+m+1)}$ and $A \in \Z_{\geq 0}^{(n+1) \times(n+m+1)}$ such that the system $C \cdot (x+\mb 1)^A=0$ is convenient and non-degenerate at $\mb 0$ with local intersection multiplicity $\mu$.
Then
$$
\mu \leq (n+m)^{n}
.$$
\end{corollary}
\begin{proof}
This follows from Theorem~\ref{T:t_{i}} noting that $t_{\ell} \leq m+n$ for $\ell=1,\ldots,n$.
\end{proof}

When $n=1$, any non-zero polynomial is convenient and non-degenerate and we recover the known bound that $\mu \le m+1 = N-1$.
When $n=2$ the bound $(n+m)^{n}$ in Corollary~\ref{C:t_{i}bis} equals $(m+2)^{2}$. We show that this bound can be improved. 

\begin{proposition}
\label{P:bettern=2}
Assume $\A \subset \Z_{\ge 0}^2$  a configuration of affine dimension $n=2$ with any codimension $m \ge 1$.  Let  $C \in \C^{2\times (m+3)}$ be a coefficient matrix such that the system $C \cdot (x+\mb 1)^A=0$ is convenient and non-degenerate at $\mb 0$ with local intersection multiplicity $\mu$.
Then
$$
\mu \leq (m+1)(m+2)
.$$
\end{proposition}
\begin{proof}
By assumption the Newton diagrams $\mathcal D(F_{1})$ and $\mathcal D(F_{2})$ are convenient. Denote by $\beta_{k,\ell}$ the non-zero coordinate of the intersection point of $\mathcal D(F_{k})$ with the $\ell$-coordinate axis, for $k,\ell=1,2$. Then $\beta_{k,\ell} \leq m+2$ by Lemma~\ref{lem:axis}. Let $\ell \in \{1,2\}$ and assume $\beta_{1,\ell}=\beta_{2,\ell}=m+2$. Then using the proof of Lemma~\ref{lem:axis} we get $t_{\ell}=m+2$ and both row vectors of $C$ belong to the one-dimensional kernel of the Vandermonde matrix~\eqref{E:Vandermonde} with
$L=m+1$, which gives a contradiction since $C$ has full rank $2$. Thus, for $\ell=1,2$, we have three possibilities: $\beta_{1,\ell},\beta_{2,\ell} \leq m+1$, $\beta_{1,\ell} \leq m+1,\beta_{2,\ell}=m+2$
or $\beta_{2,\ell} \leq m+1,\beta_{1,\ell}=m+2$. Let $\Delta_{k}$ be the segment with vertices $(\beta_{k,1},0)$ and $(0,\beta_{k,2})$. By monotonicity of the mixed covolume (Proposition~\ref{decreasing property}), we get
$$\mu=\Vol^{\circ}(\Delta(F_{1}),\Delta(F_{2})) \leq \Vol^{\circ}(\Delta_{1},\Delta_{2}).$$
We want to prove that $\Vol^{\circ}(\Delta_{1},\Delta_{2}) \leq (m+1)(m+2)$. Using the monotonicity of the mixed covolume, up to permuting $\Delta_{1}$ with $\Delta_{2}$ and the coodinates axis, we are reduced to consider the following cases
\begin{enumerate}
\item $\Delta_{1}=[(m+1,0), (0,m+1)]$ and $\Delta_{2}=[(m+2,0), (0,m+2)]$ or
\item $\Delta_{1}=[(m+1,0), (0,m+2)]$ and $\Delta_{2}=[(m+2,0), (0,m+1)]$.
\end{enumerate}
A simple computation using Formula~\eqref{alternating sum} gives $\Vol^{\circ}(\Delta_{1},\Delta_{2})=(m+1)(m+2)$ in the first case and $\Vol^{\circ}(\Delta_{1},\Delta_{2})=(m+1)^{2}$ in the second case.
\end{proof}

\begin{remark}
The proof of Proposition~\ref{P:bettern=2} can be adapted to improve the general bound $(m+n)^{n}$ for any $(n,m)$ as follows.
Assume that for $k=1,\ldots,n$ the Newton diagram $\mathcal D(F_{k})$ is convenient. Denote by $\beta_{k,\ell}$ the non-zero coordinate of the intersection point of $\mathcal D(F_{k})$ with the $\ell$-coordinate axis. For any $\ell=1,\ldots,n$, let $\tilde{\beta}_{1,\ell} \leq \cdots \leq \tilde{\beta}_{n,\ell}$ be the integers $\beta_{1,\ell},\ldots, \beta_{n,\ell}$ ordered increasingly.
Then, we get $\tilde{\beta}_{r,\ell} \leq m+r$ for $r=1,\ldots,n$ using the fact the rows of $C$ belong to the kernel of a Vandermonde matrix~\eqref{E:Vandermonde} and are linearly independent.
By monotonicity of the mixed covolume, we have $\mu \leq \Vol^{\circ}(\Delta_{1},\ldots, \Delta_{n})$
where $\Delta_{k}$ is the convex hull of the points $\beta_{k,\ell} \cdot e_{\ell}$, $\ell=1,\ldots,n$. Using monotonicity of the mixed covolume, we see that, in order to bound from above $\mu$, it is sufficient to consider the case where $\tilde{\beta}_{r,\ell} =m+r$ for $r=1,\ldots,n$. We conjecture that the maximal mixed covolume $\Vol^{\circ}(\Delta_{1},\ldots, \Delta_{n})$ is $(m+1)(m+2)\cdots(m+n)$, which would lead to
$$\mu \leq (m+1)(m+2)\cdots(m+n).$$
However, we think that a sharp upper bound for $\mu$ is $m+n \choose n$.
\end{remark}

Since $C$ has maximal rank, by Gauss elimination (that is, left multiplication by an invertible matrix) we get a matrix $C$ which contains a invertible diagonal $n \times n$ submatrix. If this can be done without losing the condition of having a convenient system, we get the following bound.

\begin{corollary}
\label{C:t_{i}}
Let $C \in \C^{n\times (n+m+1)}$ and $A \in \Z_{\geq 0}^{(n+1) \times (n+m+1)}$ such that the system $C \cdot (x+\mb 1)^A=0$ is convenient and non-degenerate at $\mb 0$ with local intersection multiplicity $\mu$. If $C$ contains
an invertible diagonal $n \times n$ submatrix, then
$$
\mu \leq (m+1)^{n}
.$$
\end{corollary}
\begin{proof}
This follows from Theorem~\ref{T:t_{i}} noting that $s_{k} \leq m+1$ for $k=1,\ldots,n$ if $C$ contains such an invertible diagonal submatrix.
\end{proof}

We also get the following interesting consequence of Theorem~\ref{T:t_{i}}.

\begin{corollary}
\label{C:d}
Let $C \in \C^{n\times (n+m+1)}$ and $A \in \Z_{\geq 0}^{(n+1) \times(n+m+1)}$ such that the system $C \cdot (x+\mb1)^A=0$ is convenient and non-degenerate at $\mb 0$ with local intersection multiplicity $\mu$.
Let $d \in \Z_{>0}^n$ be such that the integer vectors $a_{0},\ldots,a_{n+m}$ are contained in $[0,d]^{n}$. Then
$$
\mu \leq d^{n}
.$$
\end{corollary}
\begin{proof}
This follows from Theorem~\ref{T:t_{i}} noting that under our assumptions we have $t_{\ell} \leq d$ for $\ell=1,\ldots,n$.
\end{proof}

\begin{remark}
The bound $d^{n}$ in Corollary~\ref{C:d} is considerably smaller than the normalized volume $n! \Vol([0,d]^{n})=n! \cdot d^{n}$ which is the BKK bound for the number
of solutions counted with multiplicities of any polynomial  system with support contained in $[0,d]^{n}$.
\end{remark} 

\begin{example}
\label{ex:squarefirst}
(Case $n=2$ and $d=m=1$ in Corollary~\ref{C:d}.)
Let $$A=\left(\begin{array}{cccc}
1 & 1 & 1 & 1 \\
0 & 1 & 0 & 1 \\
0 & 0 & 1 & 1
\end{array}
\right).$$ It is not difficult to find a coefficient matrix
$C \in \C^{2 \times 4}$ such that $\bf 1$ is an isolated solution with multiplicity $2$ of the polynomial system $C\cdot x^{A}=0$.
In fact, by Lemma~\ref{L:linear part}, it is necessary to have $\mbox{det}(A'C'^{t})=0$ where $A'=\left(
\begin{array}{ccc}
1 & 0 & 1 \\ 0 & 1 & 1
\end{array}
\right)
$
and $C'=(c_{ij})_{1 \leq i \leq 2, 1 \leq j \leq 3}$ is the submatrix of $C$ with the $0$-th column removed.
It is easy to check that under the latter condition the system $C\cdot (x+\mb1)^{A}=0$ is always either non convenient or degenerate at the origin.
This shows that the condition that the system $C \cdot (x+\mb 1)^A=0$ is convenient and non-degenerate at $\mb 0$ cannot be dropped out in the above results. 
 This also motivates the use of Gale duality for polynomial systems introduced in Section~\ref{sec:Gale}. It is not difficult to check directly that 
the maximal multiplicity at $\mb 1$ of a system supported on $\A$ is equal to $2$ (see  Example~\ref{ex:square}).
\end{example}

\section{Local Gale duality for polynomial systems}\label{sec:G}
\label{sec:Gale}

In this section we recall how we can transform the system $f_1=\cdots= f_n=0$ in~\eqref{eq:fis} or~\eqref{eq:fisC} in $n$ variables with support $\A$ of codimension $m$ into a Gale dual system $g_1=\cdots= g_m$ in $m$ variables (see \cite{BS,BS08}). In fact we will state 
a local version of this duality and its basic properties in Theorem~\ref{Gale duality for one solution}.
We assume as before that ${\bf 1}$ is an isolated solution of~\eqref{eq:fis}. 

 Let $A$ be the matrix
defined in~\eqref{eq:A}:
\[
A = \left(\begin{array}{ccc} 1 & \dots & 1 \\ a_0 & \dots & a_{n+m} \end{array}\right),
\]
which has maximal rank $n+1$.
Let $B \in \Z^{(n+m+1)\times m}$ be any Gale dual matrix of $A$. That is, $B$ is a matrix with full rank be whose column vectors lie in the integer kernel  of $A$
(and are thus a basis of it over $\QQ$).
A Gale dual matrix of the coefficient matrix $C$ is a complex matrix $D \in \C^{(m+n+1) \times m}$ whose column vectors form a basis of the (right) kernel of $C$.
Clearly, $D$ can be chosen with coefficients in any subfield of $\C$ containing the coefficients of $C$.

As we assume that $\mb 1 =(1, \dots, 1)$ is 
a solution to~\eqref{eq:fis}. Then, 
\[
{\mb 1}^{A} = (1, \dots, 1)
\]
belongs to the kernel of $C$. Therefore, there exist vectors $\delta_{1},\ldots,\delta_{n+m} \in \R^{m}$
such that a Gale dual matrix of $C$ has the form
\begin{equation}
\label{Dreduced}
D=\left(\begin{array}{cc}
1 & \delta_{0} \\
1 & \delta_{1} \\
\vdots & \vdots \\
1 & \delta_{n+m}
\end{array}
\right), \, \text{ with } \delta_0 = {\mb 0}.
\end{equation}
Not every Gale dual matrix $D$ has this particular first row and column.

\begin{definition}\label{def:reduced}
 We will say that $D$ is \emph{reduced} when it has the form~\eqref{Dreduced} with $\delta_0=0$.
 \end{definition}
  Two Gale dual matrices of $C$  are obtained from each other by right multiplication by an invertible matrix. Two reduced Gale dual matrices of $C$  are obtained from each other 
  by right multiplication by an invertible matrix with first row and  first column vectors both equal to $(1,0,\ldots,0)$.

\smallskip
Let $y=(y_{1},\ldots,y_{m})$.
We associate polynomials of degree $1$ in $y$ to the \emph{rows} of the reduced matrix $D=(d_{ij}$):
\[
p_{i}(y)=d_{i0} + \sum_{j=1}^{m} d_{ij}y_{j}=1+\<\delta_i, y\>, \quad i=0,\ldots,n+m.
\]
Note that $p_{0}(y)=1$. Choosing another reduced Gale dual matrix of $C$ amounts to do a linear change of coordinates $y' = \xi(y)$, 
 so that $\<\delta_i, y\>=\<\delta_i', y'\>$, where $\delta_i'$ is the image of $\delta_{i}$ by an invertible linear transformation for $i=0,\ldots,m+n$.

\medskip
We introduce the Gale dual rational functions
\begin{equation}\label{rational}
\varphi_{k}(y)=\prod_{i=0}^{n+m} p_{i}(y)^{b_{ik}},  \quad k=1\ldots,m.
\end{equation}
Denote by $H$ the hyperplane arrangement $\cup_{i=1}^{n+m} \{p_{i}(y)=0\}$.
Consider the solutions of the system
\begin{equation}
\label{eq:varphi}
\varphi_{k}(y)=1,\quad  k=1\ldots,m,
\end{equation}
in the complement $\C^{m} \setminus H$.
Clearing the denominators, this amounts to look at the solutions
 in $\C^{m} \setminus H$ of the polynomial system
 
\begin{equation}\label{Gale system}
g_{k}(y)= \prod_{b_{ik}>0}p_{i}(y)^{b_{ik}} - 
\prod_{b_{ik}<0}p_{i}(y)^{-b_{ik}} =0, \quad  k=1\ldots,m.
\end{equation}
Note that the origin is a solution of~\eqref{Gale system} and~\eqref{eq:varphi} since $p_i(0) = 1$ for all $i$.

The following result follows essentially from \cite[Theorem~2.1]{BS08}. The only difference is that here
we don't require that the greatest common divisor of the maximal minors of $A$ nor the greatest common divisor of
the maximal minors of $B$ equal $\pm 1$ as in the global case when we translate all the solutions in the torus.

\begin{theorem} \label{Gale duality for one solution}
Assume that $\A=\{a_{0},a_{1},\ldots,a_{n+m}\} \subset \Z^n$ and let $D$ be a reduced Gale dual matrix of the coefficient matrix $C$. Then,
system~\eqref{eq:fis} has an isolated solution at $\mb 1$ 
of multiplicity $\mu$ if and only if the system
$g_{1}(y)=\cdots=g_{m}(y)=0$
has an isolated solution at the origin of the same multiplicity $\mu$. As a consequence,
the multiplicity at the origin of $g_{1}(y)=\cdots=g_{m}(y)=0$ depends neither on the choice of the reduced matrix $D$ nor on the choice of $B$.
\end{theorem}
\begin{proof}
Multiplying each equation of ~\eqref{eq:fis} by $x^{-a_{0}}$ does not change the multiplicity at $\mb 1$, so
we can assume that $a_0=(0,\ldots,0)$. As we mentioned before, if the subgroup  
$\ZZ \A=\Z a_{1}+\cdots+ \Z a_{n+m}$ of $\ZZ^{n}$ generated by $ \A$ is equal to $\ZZ^{n}$ and the columns of $B$ are a $\ZZ$-basis of the integer kernel of $A$,  
the result follows from \cite[Theorem~2.1]{BS08}. 
We expand this a bit further. For any $x\in (\C^{*})^n$ such that $f_{1}(x) = \dots = f_{n}(x) = 0$, 
there exists a unique $y \in \C^{m} \setminus \cup_{i=1}^{n+m} \{p_{i}=0\}$ such that $x^{a_{i}}=p_{i}(y)$ for $i=0,\ldots,n+m$,
because the columns of $D$ are a basis of the kernel of $C$ and $x^{a_{0}}=1$.
This defines a map $\psi: x \mapsto y=\psi(x)$ from the set of solutions in the torus $(\C^{*})^n$ of $f_{1}= \dots = f_n= 0$ to the solutions in the torus
$\C^{m} \setminus \cup_{i=1}^{n+m} \{p_{i}=0\}$ of~\eqref{Gale system}. Then, by \cite[Theorem 2.1]{BS08}, 
the map $\psi$ is a bijection which preserves the multiplicities of the solutions, and the result follows then from the fact that $\psi (\mb 1)=(0,\ldots,0)$.

Assume now that $\ZZ \A \neq \Z^{n}$. Since $a_{0},\ldots, a_{n+m}$ are not contained in some hyperplane $\ZZ \A$ has full rank $n$. By the
invariant factor theorem for abelian groups, 
there exists a $\ZZ$-basis $(e_1,\ldots,e_{n})$ of $\ZZ^{n}$ and non-zero integers number $t_{1},\ldots, t_{n}$  such that 
$(t_{1}e_1,\ldots,t_{n}e_{n})$ is a $\ZZ$-basis of $\Z \A$. After a change of 
coordinates if necessary, we might assume without loss of generality that $(e_1,\ldots,e_{n})$ is the standard basis. Then, 
for $j=0,\ldots,n+m$ and $i=1,\ldots, n$, the quantity $\alpha_{ij}=\frac{a_{ij}}{t_i}$ is an integer, and setting 
$u_{i}=x_{i}^{t_{i}}$ we have $x^{a_{j}}=u^{\alpha_j}$ for $j=0,\ldots,n+m$, where 
$\alpha_j=(\alpha_{1j},\ldots,\alpha_{nj})$ and $u=(u_{1},\ldots,u_{n})$. Therefore, replacing $x^{a_{j}}$ by 
$u^{\alpha_j}$, the system $f_{1}= \dots = f_{n} = 0$ becomes a system with unknowns $u$, with the same 
coefficient matrix and with support $\{\alpha_{0}=0,\alpha_{1},\ldots,\alpha_{n+m}\} \subset \ZZ^{n}$ which satisfies $\Z \alpha_{1}+\cdots+ \Z \alpha_{n+m}=\ZZ^{n}$.
We are in the previous situation, and thus this system has a root at $\mb 1$ of multiplicity $\mu$ if and only if the system
$g_{1}(y)=\cdots=g_{m}(y)=0$
has a root at $(0,\ldots,0)$ of multiplicity $\mu$. To conclude, it remains to note that $\mb 1$ is root of multiplicity $1$ of the system $1=x_{i}^{t_{i}}$, $i=1,\ldots,n$.
The argument in case the $\ZZ$-span of the columns of $B$ is not $\ker_\Z(A)$ is similar. 
\end{proof}

\begin{remark}\label{rem:ND}
The fact that the multiplicity at the origin of $g_{1}(y)=\cdots=g_{m}(y)=0$ does not depend on the choice of $D$ can be seen directly since a different
 choice of a reduced Gale dual matrix $D$ of $C$ corresponds to a different choice of linear coordinates $y=(y_{1},\ldots,y_{m})$ and 
 the multiplicity at the origin is preserved by the composition on the right by a linear isomorphism.
Note, however, that the Newton diagrams of $g_{1},\ldots,g_{m}$ are dependent on the choice of linear coordinates and thus on the choice of $D$. They also depend on the choice of $B$.
\end{remark}

The following basic results will be useful. Consider the functions  $\varphi_{k}$ and $g_k$ defined in~\eqref{rational} and~\eqref{Gale system}.

\begin{lemma}\label{remark}
 Let $\beta \in  {\ZZ}_{\geq 0}^m$.
Then,  $(\partial^{\alpha}g_{k})(0)=0$ for any non-zero $\alpha \in  {\ZZ}_{\geq 0}^m$ such that
$\alpha \leq \beta$ if and only if
$(\partial^{\alpha}\varphi_{k})(0)=0$ for any non-zero $\alpha \in  {\ZZ}_{\geq 0}^m$ such that
$\alpha \leq \beta$.
\end{lemma}
\begin{proof}
Set $P=\prod_{b_{ik}>0}p_{i}^{b_{ik}}$ and
$Q=\prod_{b_{ik}<0}p_{i}^{-b_{ik}}$. Then, $\varphi_{k}=\frac{P}{Q}$ and $g_{k}=P-Q=Q \cdot (\varphi_{k}-1)$, and the result follows from the Leibniz rule since $Q(0)=1$.
\end{proof}

Recall that $e_j$ is the $j$-th vector in the standard basis, so $\partial^{e_{j}}$ stands for the partial derivative with respect to $y_{j}$ and 
we will often use the shorter notation $\partial_{j}$ for this derivative. Consider the rational functions $\varphi_{k,j}$ defined by the equality
$$ \varphi_{k,j}(y) \, = \,  \frac {(\partial_{j}\varphi_{k})(y)}{\varphi_{k}(y)}.$$
Explicitely, we have

\begin{equation}\label{der}
\varphi_{k,j}(y)=\sum_{i=1}^{n+m}
b_{ik} \cdot \frac{d_{ij}}
{p_i(y)}.
\end{equation}

\begin{lemma}\label{useful}
Let $\beta \in  {\ZZ}_{\geq 0}^{m}$ and assume $e_{j} \leq \beta $ for some $j \in \{1,\ldots,m\}$.
Then $(\partial^{\alpha}g_{k})(0)=0$ for any $\alpha \in  {\ZZ}_{\geq 0}^m$ such that
$e_j \leq \alpha \leq \beta$ if and only if
$$
(\partial^{\alpha-e_{j}}\varphi_{k,j})(0)=0$$ for any $\alpha$ such that
$e_j \leq \alpha\leq \beta$.
\end{lemma}

\begin{proof}
Note that $d_{ij}=(\partial_{j}p_{i})(0)$. The result follows then from the Leibniz rule using that
$(\partial_{j}\varphi_{k})(y)=\varphi_{k}(y) \cdot \varphi_{k,j}(y)$ and  $\varphi_{k}(0)=1$.
\end{proof}

We now introduce the ''dual'' version of the sums $L_k$ in Definition~\ref{def:LA}.

\begin{definition}\label{def:LB}
Given a Gale dual matrix $B$ of $A$ and a reduced Gale dual matrix $D$ of $C$ as in~\eqref{Dreduced}, define
\begin{equation}\label{L_{k}}
L^*_{{k}}(\alpha)= \sum_{i=0}^{n+m} b_{ik} \delta_{i}^{\alpha}.\end{equation}
\end{definition}

Thus,  $L^*_{{k}}(\alpha)= \sum_{i=1}^{n+m} b_{ik} d_{i1}^{\alpha_{1}}d_{i2}^{\alpha_{2}} \cdots d_{im}^{\alpha_{m}}$ if $\alpha=(\alpha_{1},\ldots,\alpha_{m})$.

\smallskip
 
 The following result is obvious noting that $p_{i}(0)=1$ for all $i=1,\ldots,n+m$.

\begin{lemma}\label{Lderiv}
Assume $D$ is reduced. For any $j=1,\ldots,m$
and any $\alpha \in {\ZZ}_{\geq 0}^{m}$ such that $e_{j} \leq \alpha$,
we have
\begin{equation}\label{Lder}
(\partial^{\alpha-e_{j}}\varphi_{k,j})(0)=(-1)^{|\alpha|-1} \cdot (|\alpha|-1)! \cdot L^*_{k}(\alpha),
\end{equation}
where $|\alpha|=\sum_{i=1}^{m} |\alpha_{i}|$.
\end{lemma}

By combining the two previous lemmas, we obtain the following result.

\begin{proposition}\label{diagram}
Assume $D$ is reduced and let $\beta$ be any non zero vector  in ${\ZZ}_{\geq 0}^{m}$.
Then $(\partial^{\alpha}g_{k})(0)=0$ for any $\alpha \in  {\ZZ}_{\geq 0}^m$ such that $\alpha < \beta$ if and only if
$L^*_{k}(\alpha)=0$ for any $\alpha \in  {\ZZ}_{\geq 0}^m$ such that $\alpha < \beta$. In that case,
$(\partial^{\beta}g_{k})(0)=(-1)^{|\beta|-1} \cdot (|\beta|-1)! \cdot L^*_{{k}}(\beta)$. In particular, a point $\beta \in {\ZZ}_{\geq 0}^m$ is a vertex of the Newton diagram
of $g_{k}$ if and only if $L^*_{k}(\beta) \neq 0$ and $L^*_{k}(\alpha)= 0$ for any $\alpha \in  {\ZZ}_{\geq 0}^m$ such that $\alpha < \beta$. Moreover, for any point $\beta \in {\ZZ}_{\geq 0}^m$
lying on the Newton diagram of $g_{k}$ the coefficient of $y^{\beta}$ in $g_{k}$ equals
$(-1)^{|\beta|-1} (|\beta|-1)! \frac{L^*_{k}(\beta)}{\beta_1!\dots\beta_m!}$.
\end{proposition}

\begin{proof}
We first prove that $(\partial^{\alpha}g_{k})(0)=0$ for any non-zero $\alpha \in \Z_{\geq 0}^m$ such that $\alpha \leq \beta$ if and only if
 $L^*_{k}(\alpha)=0$ for any non-zero $\alpha \in \Z_{\geq 0}^m$ such that $\alpha \leq \beta$.
Let $j \in \{1,\ldots,m\}$ such that $e_{j} \leq \beta$. 
Applying Lemma~\ref{useful} and Lemma~\ref{Lderiv} we get $(\partial^{\alpha}g_{k})(0)=0$ for any $\alpha \in  {\ZZ}_{\geq 0}^m$ such that $e_{j } \leq \alpha \leq \beta$ if and only if
$L^*_{k}(\alpha)=0$ for any $\alpha \in  {\ZZ}_{\geq 0}^m$ such that $e_{j } \leq \alpha \leq \beta$. Collecting this for all $j$ such that 
$e_{j} \leq \beta$ leads to the result. From the previous computation, Lemma~\ref{lem:tor} and using that $g_{k}=P-Q=Q \cdot (\varphi_{k}-1)$
we deduce that if $L^*_{k}(\alpha)=0$ for any non-zero $\alpha \in \Z_{\geq 0}^m$ such that $\alpha <\beta$, 
then $(\partial^{\beta}g_{k})(0)=(-1)^{|\beta|-1} \cdot (|\beta|-1)! \cdot L^*_{k}(\beta)$.
\end{proof}

\section{$H$-duality} \label{sec:H}
\label{S:H}

In this section we introduce the notion of $H$-duality of a sparse polynomial system and we relate it to the notion of Gale duality that we recalled in the previous section.
We prove in Theorem~\ref{th:dual} the equality of the Newton diagrams at $\mb 0 \in \C^m$. In Subsection~\ref{ssec:resume} we highlight the main relations.

By Theorem~\ref{Gale duality for one solution}, the systems $C \cdot (x+ \mb 1)^A=0$ and the Gale dual system~\eqref{Gale system} have the same multiplicity at the origin
when $\mb 1$ is an isolated solution of  $C \cdot x^A=0$.
Given a complex number $d$, we consider the analytic function $x^d: \C \setminus \R_{\le 0} \to \C$ defined by 
\[ x^d = e^{d \,  {\rm log}(x)},\]
where ${\rm log}$ is the branch of the logarithm defined in $ \C \setminus \R_{\le 0}$ that taked the value $0$ at $x=1$.  
The function $(x+ \mb 1)^d$ is analytic around $0$, where it takes the value $1$.
When $d \in \Z$, this is just the Laurent polynomial $(x+ \mb 1)^d$.  Given ${\bf d} \in \C^n$, we can similarly define the analytic
function $x^{\bf d}: (\C \setminus \R_{\le 0})^n \to \C$ by $x^{\bf d}= \prod_{i=1}^n x_i^{d_i}$ and substitute $x = x+ \mb 1$
coordinatewise.

Given a Gale dual matrix $B$ of $A$ and a reduced Gale dual matrix $D$ of $C$ as in~\eqref{Dreduced}, we define the system of analytic functions:
\begin{equation}
\label{eq:his}
h_k(y) \, = \sum_{i=0}^{m+n} b_{ik} y^{\delta_i} \, = \, 0, \, k = 1, \dots, m.  
\end{equation}
We will abbreviate system~\eqref{eq:his} as $B^t \cdot y^{D^t}=0$.  Note that this is a system in codimension many variables.
We denote
\begin{equation} \label{eq:H}
H_k(y)= h_k(y+ {\mb 1}) \, = \, 0, \, k=1, \dots, m.
\end{equation}
We will say that the system $B^t \cdot (y+\mb1)^{D^t}=0$ given by~\ref{eq:H} \emph{is $H$-dual }to the system $C \cdot (x+\mb1)^A=0$. Note that $C \cdot (x+\mb1)^A=0$ is $H$-dual to
$B^t \cdot (y+\mb1)^{D^t}=0$. 

\medskip

The next result follows from the observation that the quantity $L^*_{k}$ defined in~\eqref{L_{k}}
coincides with the quantity~\eqref{L_{f}} associated to $h_{k}$ (instead of $f_{k}$).

\begin{theorem}
\label{th:dual} 
Assume that $\mb 1 \in \C^n$ is an isolated solution of the system $C \cdot x^A=0$. Let $B$ be a $\Z$-Gale dual matrix of $A$ and $D$ be a reduced Gale dual matrix of $C$.
Then the polynomials $g_{1}(y),\ldots,g_{m}(y)$ defining the Gale dual system and the polynomials $H_1, \dots, H_m$ defining the $H$-dual system
\begin{equation}\label{eq:BD}
B^{t} \cdot (y+\mb 1)^{D^t}=0
\end{equation}  
 have respectively the same Newton diagrams. 
Moreover, for any $\beta  \in {\ZZ}_{\geq 0}^m$ which lies on $\mathcal D(H_k)$, the coefficient of $y^\beta$ in $H_k$ equals $\frac{L^*_{k}(\beta)}{\beta_1!\dots\beta_m!}$ while the
coefficient of $y^\beta$ in $g_k$ equals $(-1)^{|\beta|-1} (|\beta|-1)! \frac{L^*_{k}(\beta)}{\beta_1!\dots\beta_m!}$.

\smallskip

If the Gale dual system and the system~\eqref{eq:BD} are convenient and are non-degenerate at $\mb 0$, then all three systems
$C \cdot (x+{\mb 1})^A=0$, $g_{1}(y)=\cdots=g_{m}(y)=0$ and $B^{t} \cdot (y+{\mb 1})^{D^t}=0$ have the same multiplicity at the origin.
 \end{theorem} 
 
\begin{proof}
Note that $0 \in \C^m$ is a solution of the system $B^{t} \cdot (y+\mb 1)^{D^t}=0$ since the sum of the rows of $B$ is zero.
The same is true for the Gale dual system $g_{1}(y)=\cdots=g_{m}(y)=0$ due to Theorem~\ref{Gale duality for one solution} and the fact that  $\mb 1$ is a solution of $C \cdot x^A=0$.
The $k$-th polynomial of the system $B^{t} \cdot x^{D^t}=0$ is the polynomial $h_k(x)$ defined in~\eqref{eq:his}. Observe that $L^*_{k}(\alpha)$
defined in~\eqref{L_{k}} is equal to the toric derivative $\theta^{\alpha}h_k$ evaluated at the point ${\mb 1} \in \C^m$, where as before
$\theta^{\alpha}=\prod_{i=1}^m \theta_i^{\alpha_i}$ and $\theta_i=x_{i} \frac{\partial}{dx_{i}}$. 
Using Lemma~\ref{lem:tor}, we get that $\beta \in \Z_{\geq 0}^m$ is a vertex of the Newton diagram of $H_{k}$ if and only if $L^*_{k}(\beta) \neq 0$ and $L^*_{k}(\alpha)$ for any  non-zero $\alpha \in \Z_{\geq 0}^m$
such that $\alpha < \beta$. Moreover, for any $\beta  \in {\ZZ}_{\geq 0}^m$ which lies on $\mathcal D(H_k)$, the coefficient of $y^\beta$ in $H_k$ equals $\frac{L^*_{k}(\beta)}{\beta_1!\dots\beta_m!}$.
Proposition~\ref{diagram} gives the analogous results for $g_{k}$.

Finally, if the Gale dual system and the system~\eqref{eq:BD} are convenient and non-degenerate at $\mb 0$, 
then they have the same multiplicity at the origin which equals the mixed covolume of their common Newton diagrams by 
Theorem~\ref{main co-mixed}. Moreover, the systems $C \cdot (x+{\mb 1})^A=0$ and $g_{1}(y)=\cdots=g_{m}(y)=0$ have the same multiplicity at the origin by
Theorem~\ref{Gale duality for one solution}.
\end{proof}

\begin{remark} The final assertion of Theorem~\ref{th:dual} is a local statement. 
Assuming $C$ has rational coefficients and $D$ has integer coefficients, one might ask if the systems $C \cdot (x+\mb 1)^A=0$ and $B^{t} \cdot (y+\mb 1)^{D^t}=0$ in 
Theorem~\ref{th:dual} have the same number of non-zero solutions.  This is not the case.
A simple example can be found already in the case $n=m=1$. Take $\A=\{0,1,3\}$ and  $B=(2, -3,1)^{t}$. 
Take $C=(1, -2,1)$ so that the matrix $D$ with first column $(1,1,1)^{t}$ and second column $(0,1,2)^{t}$ is Gale dual to $C$. 
Then the system $C \cdot x^A=0$  equals the degree $3$ polynomial $1-2x+x^{3}=0$, which has three simple (real) roots $1$ and $\frac{-1\pm \sqrt{5}}{2}$.
On the other hand the system $B^{t} \cdot y^{D^t}=0$ equals the degree $2$ polynomial $2-3x+x^2 =(x-1)(x-2)=0$, with two (real) roots.
\end{remark}

\smallskip

\subsection{Viewing all the systems together} \label{ssec:resume}

Start with a system $C \cdot (x+\mb 1)^A=0$ of dimension $n$ and codimension $m$ and having $\mb 0$ as an isolated solution 
of multiplicity $\mu \geq 1$. Assume that a Gale dual matrix $B$ for $A$ and a reduced Gale dual matrix $D$ for $C$ are given.
We denote the associated Gale system $g_{1}(y)=\cdots=g_{m}(y)=0$ by the compact form $(D \cdot y) ^{B}=1$.
 This is a system in $m$ variables and it has $\mb 0$ as a solution with the same multiplicity $\mu$.
The $H$-dual system $B^{t} \cdot (y+\mb 1)^{D^t}=0$ has the same Newton diagrams than the Gale system. Its number of variables (dimension) is $m$ while its codimension is $n$.
It has $\mb 0$ as a solution of multiplicity $\mu' \geq 1$. If the Gale system $(D \cdot y) ^{B}=1$ is convenient and non-degenerate, then $\mu' \geq \mu$. On the other hand,
if $B^{t} \cdot (y+\mb 1)^{D^t}=0$ is convenient and non-degenerate, then $\mu \geq \mu'$. 
One can also consider a polynomial system Gale dual to $B^{t} \cdot (y+\mb 1)^{D^t}=0$. Assume now that $a_{0}=(0,\ldots,0)$. 
Then the matrix $A^{t}$ is a reduced Gale dual matrix of $B^{t}$. Moreover $C^{t}$ is Gale dual to $D^{t}$.
 Thus a polynomial system Gale dual to $B^{t} \cdot (y+\mb 1)^{D^t}=0$ is $(A^t \cdot x)^{C^{t}}=1$. 
 So again, if $(A^t \cdot x)^{C^{t}}=1$ is convenient and non-degenerate then $\mu \geq \mu'$ while if $C \cdot (x+\mb 1)^A=0$ 
 is convenient and non-degenerate then $\mu' \geq \mu$. We resume this in the following diagram
\bigskip

\begin{equation}\label{eq:resume}
\begin{array}{cccl}
C \cdot (x+\mb 1)^A=0 & \stackrel{\tiny \mbox{Gale map}}\longrightarrow  &(D \cdot y) ^{B}=1 & \quad \mbox{mult. $\mu$}\\
\Big\updownarrow & & \Big\updownarrow & \\
(A^t \cdot x)^{C^{t}}=1 &  \stackrel{\tiny \mbox{Gale map}}\longleftarrow  & B^{t} \cdot (y+\mb 1)^{D^t}=0 & \quad \mbox{mult. $\mu'$}\
\end{array}
\end{equation}
\bigskip

The horizontal arrows represent Gale duality map for polynomial systems.
The vertical double arrows indicate polynomial systems which have the same Newton diagrams, and thus the same number of variables. Systems on the first column have $n$ variables
while those on the second column have $m$ variables.
The multiplicity at $\mb 0$ of the systems on the first row equals $\mu$, while the multiplicity at $\mb 0$ of the systems on the second row equals $\mu'$.
If at least one system on the first row is convenient and non-degenerate at $\mb 0$, then $\mu \leq \mu'$.
If at least one system on the second row is convenient and non-degenerate at $\mb 0$,
then $\mu \geq \mu'$.

\medskip

Recall that two matrices are said left (resp., right) equivalent if one matrix is obtained from the other by left (resp., right) multiplication by an invertible matrix.
If $\tilde{C}$ and $\tilde{A}$ are left equivalent to $C$ and $A$, respectively then  $C \cdot (x+\mb 1)^{A}=0$ and $\tilde{C} \cdot (x+\mb 1)^{\tilde{A}}=0$ have the same multiplicity at $\mb 0$, but not necessarily the same Newton diagrams. On the other side, if $\tilde{B}$ and $\tilde{D}$ are right equivalent to $B$ and $D$, respectively, with both $D$ and $\tilde{D}$ reduced, then
$(D \cdot y) ^{B}=1$ and $(\tilde{D} \cdot y) ^{\tilde{B}}=1$ have the same multiplicity at $\mb 0$, but not necessarily the same Newton diagrams.
Similar statements hold true for the systems on the second row of~\eqref{eq:resume}.
Note that taking transpose of matrices transforms left equivalence to right equivalence (and vice versa). Note also that left and right equivalences behave well with respect to Gale duality for matrices in the sense that if $\tilde{M}$ is left equivalent to $M$ and $\tilde{N}$ is right equivalent to $N$ then $\tilde N$ is a Gale dual matrix for $\tilde M$ if and only if $N$ is a Gale dual matrix for $M$. Therefore,~\eqref{eq:resume} holds true if if we replace $A,C$ by left equivalent matrices and $B,D$ by right equivalent matrices (still imposing that $A^{t}$ and $D$ are reduced). Under the assumption that $\mb 1$ is an isolated solution and the matrices have full rank, we prove in Proposition~\ref{prop:conv}, Proposition~\ref{prop:convenient}, Proposition~\ref{prop:conv*}, and Proposition~\ref{prop:convenient*} that there always exist equivalent matrices so that all polynomial systems in~\eqref{eq:resume} are convenient. This lead us to pose Question~\ref{q:ndc} in the Introduction.
When $n=1$ or $m=1$, Question~\ref{q:ndc} has a positive answer because the Newton diagram of any non-zero univariate polynomial vanishing at zero is convenient and non-degenerate at zero.
A positive answer to Question~\ref{q:ndc} in general would imply that we always have $\mu=\mu'$ in~\eqref{eq:resume} and would imply a positive answer to Question~\ref{q:Mnm}.

\begin{example}
\label{firstE}
Consider the system $(x-1)^3=0$ in dimension $n=1$ with codimension $m=2$.  Clearly $1$ is an isolated zero with multiplicity $3$ and any non-zero univariate polynomial is convenient and non-degenerate.  We will use this simple case to explore in this example and the following one the different systems in~\eqref{eq:resume}.

Setting $a_i=i$, $i=0\ldots, 3$, we get the matrices
$$A=
\left(
\begin{array}{cccc}
1 & 1 & 1 & 1 \\
0 & 1 & 2 & 3
\end{array}
\right),
\quad 
C=\left(
\begin{array}{cccc}
-1 & 3 & -3 & 1
\end{array}
\right).$$
The matrix $B$ below is a Gale dual matrix of $A$ and the matrix $D$ is  a reduced  Gale dual matrix of $C$:
$$
B=
\left(
\begin{array}{rr}
1 & 2 \\
-2 & -3 \\
1  & 0 \\
0   & 1
 \end{array}
\right), 
\quad
D=
\left(
\begin{array}{rrr}
1 & 0 & 0\\
1 & 1 & 0\\
1  & 0 & 1\\
1   & -3 & 3 \\
 \end{array}
\right).
$$
Any other Gale dual matrix of $A$ is obtained by right multiplication of $B$ with an integer matrix with non-zero determinant.
Similarly, any other reduced Gale dual matrix of $C$ is obtained by right multiplication of $D$ with a invertible matrix with real coefficients
and first row and column vectors both equal to $(1,0,0)$. Under our choice of $B$ and $D$, we have $p_0(y)=1$, $p_{1}(y)=1+y_{1}$, $p_2(y)=1+y_{2}$ and $p_{3}(y)=1-3y_{1}+3y_{2}$.
The associated Gale system is $g_{1}(y)=g_{2}(y)=0$ with $g_{1}(y)=p_{2}(y)-p_1^{2}(y)=y_{2}-2y_{1}-y_{1}^{2}$ and $g_{2}(y)=p_{3}(y)-p_1^{3}(y)=3y_{2}-6y_{1}-3y_{1}^{2}-y_{1}^{3}$. Thus $g_{1}$ and $g_{2}$ have the same Newton diagram, given by the segment with vertices $(1,0)$ and $(0,1)$, and
the corresponding truncated polynomials are $y_{2}-2y_{1}$ and $3y_{2}-6y_{1}$ (these truncated polynomials are the polynomials $g_{1}^{\epsilon}$ and $g_{2}^{\epsilon}$ for $\epsilon=(1,1)$).
So the Gale system is convenient but it is degenerate at the origin since $y_{2}-2y_{1}=3y_{2}-6y_{1}=0$ has a solution in $(\C^*)^2$.
In fact, since the Gale system is convenient it should be degenerate at the origin for otherwise the multiplicity 
of $0$ would be equal to the corresponding mixed covolume (Theorem~\ref{main co-mixed}), which equals $1$
 in that case, but we know that the multiplicity at the origin of the Gale system is equal to the multiplicity at $\bf 1$ 
 of the initial system, which equals $3$ by Theorem~\ref{Gale duality for one solution}. Let us choose another reduced Gale dual matrix for $C$. Multiplying on the right the starting matrix $D$ with an invertible matrix
$$\left(
\begin{array}{ccc}
1& 0 & 0 \\
0 & a & b \\
0 & c & d
\end{array}
\right)$$
gives the reduced Gale dual matrix of $C$
$$
\left(
\begin{array}{ccc}
1 & 0 & 0\\
1 & a & b\\
1  & c & d\\
1   & 3(c-a) & 3(d-b) \\
 \end{array}
\right),
$$
again denoted by $D$. Then, $p_0(y)=1$, $p_{1}(y)=1+ay_{1}+by_{2}$, $p_2(y)=1+cy_1+dy_{2}$ and $p_{3}(y)=1+3(c-a)y_{1}+3(d-b)y_{2}$,
which gives
$$g_{1}(y)=(c-2a)y_{1}+(d-2b)y_{2}-(ay_{1}+by_{2})^{2}$$
and
$$g_2(y)=3(c-2a)y_{1}+3(d-2b)y_{2}-3(ay_{1}+by_{2})^{2}-(ay_{1}+by_{2})^{3}.$$
If $c \neq 2a$ and $d \neq 2b$, then the Newton diagrams are as before both equal to the segment with vertices $(1,0)$ and $(0,1)$ and we conclude again that the system is convenient
and degenerate at the origin. Assume from now on that $c=2a$. Then $d \neq 2b$ and $a \neq 0$ since $ad-bc \neq 0$. 
Thus the Newton diagrams of $g_{1}$ and $g_{2}$ are both equal to the segment with vertices $(2,0)$ and $(0,1)$.
The corresponding truncated system is $(d-2b)y_{2}-a^{2}y_{1}^2=3(d-2b)y_{2}-3a^{2}y_{1}^2=0$ and so again
 the system $g_{1}=g_{2}=0$ is convenient and degenerate at the origin. In fact this was expected, since otherwise the 
 multiplicity at the origin would have been equal to the corresponding mixed covolume which equals $2$ and not $3$.
The system $g_{1}=g_{2}-3g_{1}=0$ is equivalent to the Gale system, moreover it is non-degenerate at the origin.
 It is easy to compute, using Formula~\ref{alternating sum}, that the corresponding mixed covolume is equal to $3$ as expected. Note however that $g_{1}=g_{2}-3g_{1}=0$ is not a Gale dual system of $(x-1)^3=0$.
Using Theorem~\ref{th:dual}, we get that the Gale system $g_{1}=g_{2}=0$ and the system
$B^{t} \cdot (y+1)^{D^t}=0$ have the same Newton diagrams (the segment with vertices $(2,0)$ and $(0,1)$ for both equations). Moreover,
the truncated system corresponding to
$B^{t} \cdot (y+1)^{D^t}=0$ is $(d-2b)y_{2}+a^{2}y_{1}^2=3(d-2b)y_{2}+3a^{2}y_{1}^2=0$.
So by substracting three times the first column of $B$ from the second column of $B$, we will cancel the coefficients
 in front of $y_{2}$ and $y_{1}^2$ in the second equation of $B^{t} \cdot (y+1)^{D^t}=0$ which will raise the Newton diagram of the second equation higher, for this system and for the Gale system as well.
Explicitely, keeping the matrix $D$ as before with $c=2a$ and taking as for $B$ the matrix
$$
B=
\left(
\begin{array}{rr}
1 & -1 \\
-2 & 3 \\
1  & -3 \\
0   & 1
 \end{array}
\right)$$
we get a Gale system where the Newton diagram of $g_{1}$ and the corresponding truncated polynomial is unchanged. 
We compute that the Newton diagram of $g_{2}$ is the union of the segments $E_{1}=[(0,2), (1,1)]$ and $E_{2}=[(1,1), (3,0)]$, 
assuming that $6b^{2}+3d^{2}-9bd \neq 0$. In fact the truncation of $g_{2}$ to $E_{1}$ is the polynomial
$-(6b^{2}+3d^{2}-9bd)y_{2}^{2}+3a(2b-d)y_{1}y_{2}$ while the truncation of $g_{2}$ to $E_{2}$ is the polynomial 
$3a(2b-d)y_{1}y_{2}+2a^{3}y_{1}^{3}$. We conclude that this Gale system is convenient and non-degenerate at the origin. 
Thus the multiplicity at the origin is equal to the corresponding mixed covolume
(Theorem~\ref{main co-mixed}). Using Formula~\ref{alternating sum}, we compute that this mixed covolume equals $3$, as expected.
\end{example}

\begin{example}
\label{firstEbis}
We use $H$-dual systems $B^{t} \cdot (y+\mb 1)^{D^t}=0$ coming from Example \ref{firstE} as starting systems $C \cdot (x+1)^A=0$ .
Consider the system $C \cdot (x+1)^A=0$ where $C$ is the transpose of 
the matrix $B$ considered at the beginning of Example \ref{firstE},  and $A$ is the transpose of the matrix $D$ considered at the beginning of Example \ref{firstE}:
$$
C=\left(
\begin{array}{rrrr}
1& -2 & 1 & 0 \\
2 & -3 & 0 & 1
\end{array}
\right),
\quad
A=\left(
\begin{array}{rrrr}
1 & 1 & 1 & 1 \\
0 & 1 & 0 & -3 \\
0 & 0 & 1 & 3
 \end{array}
\right).$$
Using the computations made in Example  \ref{firstE} and Theorem \ref{th:dual}, we get that this system is convenient but degenerate at the origin.
Now consider the following left equivalent matrices to $C$ and $A$ respectively:
$$
\tilde{C}=\left(
\begin{array}{rrrr}
1& -2 & 1 & 0 \\
-1& 3 & -3 & 1
\end{array}
\right), \quad 
\tilde{A}=\left(
\begin{array}{rrrr}
1 & 1 & 1 & 1 \\
0 & a & c & 3(c-a) \\
0 & b & d & 3(d-b)
 \end{array}
\right).$$
(transpose of the final matrices $B$ and $D$ in Example  \ref{firstE}).
Then using the computations made in Example \ref{firstE} and Theorem \ref{th:dual}, we get that the system $\tilde{C} \cdot (x+1)^{\tilde{A}}=0$ is convenient and non-degenerate at the origin if $c=2a$.
\end{example}

\medskip

We end this section with an easy observation  to detect in terms of the matrices when the intersection multiplicity is higher than $1$.
Assume that $a_0=(0,\ldots,0)$ and let $A'$ be the matrix with column vectors $a_{1},\ldots, a_{m+n}$ in this order
($A'$ is obtained by removing the first row and
 $0$-th column of $A$) and let $C'$ be the matrix $C$ with $0$-th column removed.
Given any reduced Gale matrix D, let $b=(b_{0},\ldots,b_{m+n})$ be any column vector of $B$. The corresponding Gale polynomial is
$g(y)=\prod_{b_{i}>0} (1+\langle \delta_{i}, y \rangle)^{b_{i}}-\prod_{b_{i}<0} (1+\langle \delta_{i}, y \rangle)^{-b_{i}}$ and we get
$$
g(y)=\sum_{i=1}^{m+n} b_{i} \langle \delta_{i}, y \rangle+ \mbox{terms of degree $\geq 2$}.
$$
So the linear part of the Gale system $g_1(y)=\cdots=g_{m}(y)=0$ equals
\begin{equation}\label{linear}
\sum_{i=1}^{m+n} b_{ij} \langle \delta_{i}, y \rangle= \sum_{i=0}^{m+n} b_{ij} \langle \delta_{i}, y \rangle=0, \quad j=1,\ldots,m,
\end{equation}
because $\delta_0=0$. It can be written as
$$
(B'^{t}D') \cdot y=0,$$
where  $y$ is  a column vector, $B'$ denotes the submatrix of $B$ without the first row  and $D'$ is the matrix obtained by removing the first column and the first row of $D$. 
Note that~\eqref{linear} is also the linear part of the $H$-dual system~\eqref{eq:H} by Theorem~\ref{th:dual}. 

\begin{lemma}
\label{L:linear part}
Assume that ${\mb 1}$ is a solution of the system $C \cdot x^{A}=0$ with multiplicity $\mu \geq 1$. Let $B$ be any Gale dual matrix of $A$ and let $D$ be any reduced Gale dual matrix of $C$.
The following conditions are equivalent:
\begin{enumerate}
\item$\mu \geq 2$ 
\item $\mbox{det}(A' \cdot C'^t) =0$
\item $\mbox{det}(B'^{t} \cdot D') =0$.
\end{enumerate}
\end{lemma}

The proof of Lemma~\ref{L:linear part} follows from the computation of the toric jacobians of the systems and the use of the Cauchy Binet formula.

\section{Bounds for the local intersection multiplicity via duality}
\label{sec:ubound}

Assume a Gale dual matrix $B$ of $A$ and a reduced Gale dual matrix $D$ of $C$ as in~\eqref{Dreduced} are given.
Applying results of Section~\ref{sec:Aside} to the system $B^{t} \cdot (y+\mb 1)^{D^t}=0$ provides results for the Gale system
and then for $C \cdot (x+\mb 1)^{A}=0$ using Theorem~\ref{th:dual} and Theorem~\ref{Gale duality for one solution}.
We next define the dual notions of those introduced in Definition~\ref{def:rhost}, that will be used in the bounds given
in Theorem~\ref{T:t_{i}*} and its consequences Corollaries~\ref{C:t_{i}*} and~\ref{C:t_{i}bis*}.

\begin{definition}\label{def:rhost*}
Let $k, \ell \in \{1,\ldots,m\}$. For $u \in \{d_{0,\ell},\ldots, d_{m+n,\ell}\}$ 
define
\begin{equation}
\label{E:sum*}
\bar{b}_{k,\ell,u}=\sum_{d_{j, \ell}=u} b_{j k}
\end{equation}
We also set
\begin{equation}\label{E:cardCA*}
\rho_{k,\ell} ^{*}=\left| \{u \in  \{d_{0,\ell},\ldots, d_{m+n,\ell}\} \, , \bar{b}_{k,\ell,u} \neq 0\}\right|-1,
\end{equation}
\begin{equation}\label{E:cardC*}
s_k ^{*}= \left| \{  j \in \{0,\ldots, m+n\} \, , b_{jk} \neq 0\}\right|-1,
\end{equation}
and
\begin{equation}\label{E:card*}
t_\ell ^{*}= \left| \{ d_{j,\ell}, \, j=0,\ldots, m+n\}\right|-1.
\end{equation}
\end{definition}

We obviously have $\rho_{k,\ell}^* \leq \mbox{min}(s_k ^{*}, t_\ell ^{*})$.

\begin{lemma}
\label{lem:axis*}
The polynomial $g_{k}$ (resp., $H_{k}$) is identically zero on the $\ell$-coordinate axis if and only if the sum
$\bar{b}_{k,\ell,u}$ defined in~\eqref{E:sum*} vanishes for all $u \in  \{d_{0,\ell},\ldots, d_{m+n,\ell}\}$, equivalently, $\rho_{k,\ell} ^{*}=-1$.
Assume that $g_{k}$ (resp., $H_{k}$) is not identically zero on the $\ell$-coordinate axis and let 
$\beta=\beta_{k,\ell} \cdot \e_{\ell}$ be the intersection point of $\mathcal D(g_{k})$ (resp., $\mathcal D(H_{k})$) with the $\ell$-coordinate axis.
Then $\beta_{k,\ell} \le \rho_{k,\ell} ^{*} \le \mbox{min} (s_{k}^{*},t_\ell^{*})$, where $s_{k}^{*}$ and $t_{\ell}^{*}$ as in Definition~\ref{def:rhost*}. 
\end{lemma}
\begin{proof}
This is a direct consequence of Theorem~\ref{th:dual} and  Lemma~\ref{lem:axis}  applied to the system $B^{t} \cdot (y+\mb 1)^{D^t}=0$.
\end{proof}

\begin{proposition}
\label{prop:conv*}
Assume that $\bf 0$ is an isolated solution of the system Gale system $g_{1}=\cdots=g_{m}=0$. Then there exists an invertible matrix $R \in \C^{m \times m}$ such that
the Gale system associated to the Gale dual matrix $BR$ of $A$ and the reduced Gale dual matrix $D$ of $C$ is convenient,or equivalently, such that $R^{t}B^{t} \cdot (y+\mb 1)^{D^t}=0$ is convenient.
\end{proposition}
\begin{proof}
Let $\ell \in \{1,\ldots,m\}$, then there exist $k$ and $u \in \{d_{0,\ell},\ldots, d_{m+n,\ell}\}$ such that $\bar{b}_{k,\ell,u} \neq 0$ 
for otherwise $g_{k}$ would vanish identically on the $\ell$-coordinate axis by Lemma~\ref{lem:axis*}. Then arguing as in the proof of
Proposition~\ref{prop:conv}, or using  Proposition~\ref{prop:conv}  and Theorem~\ref{th:dual}, we can multiply $B$ on the left by an invertible matrix so that no polynomial of the associated Gale system vanishes identically on any coordinate axis.
\end{proof}

\begin{proposition}
\label{prop:convenient*}
There exists an invertible matrix $R \in \R^{(m+1) \times (m+1)}$ such that the Gale system associated to the Gale dual matrix $B$ of $A$ and the reduced Gale dual matrix $DR$ of $C$ is convenient, or equivalently, such that $B^{t} \cdot (y+\mb 1)^{R^{t}D^{t}}=0$ is convenient.
\end{proposition}
\begin{proof}
This follows from Proposition~\ref{prop:convenient} applied to the system $B^{t} \cdot (y+\mb 1)^{D^t}=0$ and Theorem~\ref{th:dual}.
\end{proof}

Recall that we denote by $\mu$ the multiplicity at the isolated solution $\bf 1$ of the system $C \cdot x^A=0$. This is also the multiplicity at $\bf 0$ of the Gale system $g_{1}=\cdots=g_{m}=0$
by Theorem~\ref{Gale duality for one solution}.
For $k=1,\ldots,m$ consider the polytope
\begin{equation}
\label{eq:Dpolytope*}
\Gamma_{k}^{*}=\mbox{conv} \, \{\rho_{k,\ell}^{*} \cdot \e_{\ell}, \, \ell=1,\ldots,m\}
\end{equation}
where $\rho_{k,\ell}^{*} $ is defined in~\eqref{E:cardCA*}.

\begin{theorem}
\label{T:t_{i}*}
Assume that the Gale system $g_{1}=\cdots=g_{m}=0$ is convenient and non-degenerate at $\mb 0$.
Then
$$
\mu \leq  \Vol^{\circ}(\Gamma_{1}^{*},\ldots,\Gamma_m^{*}) \leq \mbox{min}\left(\prod_{i=1}^{m} s_i^{*},\prod_{i=1}^{m} t_i^{*}\right),$$
where $s_1^*, \ldots, s_{m}^*$ and $t_1^{*},\ldots, t_{m}^{*}$ are defined in~\eqref{E:cardC*} and in~\eqref{E:card*}.
\end{theorem}
\begin{proof}
This is similar to the proof of Theorem~\ref{T:t_{i}} using  Lemma~\ref{lem:axis*}.
\end{proof}

\begin{corollary}
\label{C:t_{i}*}
Assume that the Gale system $g_{1}=\cdots=g_{m}=0$ is convenient and non-degenerate at $\mb 0$. If $B$ contains a diagonal invertible matrix of size $m\times m$  then
$$
\mu \leq (n+1)^{m}
.$$
\end{corollary}
\begin{proof}
This follows from Theorem~\ref{T:t_{i}*} noting that $s_{k}^{*} \leq n+1$ for $k=1,\ldots,m$ if $B$ contains a diagonal invertible matrix of size $m \times m$.
\end{proof}

\begin{proposition}
\label{C:t_{i}bis*}
Assume that the Gale system $g_{1}=\cdots=g_{m}=0$ is convenient and non-degenerate at $\mb 0$.
Then
$$
\mu \leq (n+m)^{m}
.$$
When $m=2$, we have the better bound $\mu \leq (n+1)(n+2)$.
\end{proposition}
\begin{proof}
This follows from Theorem~\ref{T:t_{i}*} noting that $t_{\ell}^{*} \leq m+n$ for $\ell=1,\ldots,m$. A proof of the last statement can be obtained using Theorem~\ref{th:dual}
and the proof of Proposition~\ref{P:bettern=2}. 
\end{proof}

\section{An explicit sparse system with high local multiplicity} \label{sec:lbound}

Given $n$ and $d$, consider the $n$-variate polynomials in~\eqref{eqn:poly}
and the explicit system $f_{m+1, n}(z) = f_{m+2, n}(z) = \dots = f_{m+n, n}(z) = 0$ 
from~\eqref{eq:fij}. This section is devoted to prove Theorem~\ref{conj}, together with
the refined result in Proposition~\ref{prop:n=2} in the case of dimension $2$.
The support of the system $f_{m+1, n}(z) = f_{m+2, n}(z) = \dots = f_{m+n, n}(z) = 0$ is
\begin{equation}\label{eq:polytope}
\A_{n,m} \, = \, \{ \mb 0, \mb 1, (2,4,\dots, 2^n), \dots, (m+n, (m+n)^2, \dots, (m+n)^n) \}.
\end{equation}
We denote by $A_{n,m} \in \Z^{(n+1) \times(n+m+1)}$ the associated exponent matrix
\begin{equation}\label{A}
A_{n,m}= \left( j^{i} \right)_{0 \leq i \leq n,0 \leq j \leq m+n}.\end{equation}
Note that $A_{n,m}$ has full rank $n+1$ and $N = n+m+1$ columns, so the dimension of the system is $n$ and
the codimension of the support set $\A_{n,m}$ is $m$.

The system $f_{m+1, n}(x) = f_{m+2, n}(x) = \dots = f_{m+n, n}(x) = 0$ can be written as
\begin{equation}\label{Jens}
C_{n,m} \cdot x^{A_{n,m}}=0
\end{equation}
 with
\begin{equation}\label{C}
C_{n,m}= \left({(-1)}^j {m+i \choose j} \right)_{1 \leq i \leq n,0 \leq j \leq m+n}\end{equation}
with the convention that ${a \choose b}=0$
when $b>a$.
We will need the following elementary result.

\begin{lemma}\label{L:bino}
For any integer numbers $k,\ell  \geq 0$, the sum
$$
\sum_{q=0}^{\ell} {(-1)}^{q} {\ell \choose q} q^{k}
$$
is equal to the coefficient of $x^{k}$ in the expansion as a series of
$k ! \cdot  (1-e^{x})^{\ell}$.
As a consequence this sum is equal to $0$ when $k <\ell$ and is non-zero for $k \geq \ell$.\end{lemma}

\begin{proposition}\label{Gale duals}
For any positive integers $n,m$ consider the matrices~\eqref{A} and~\eqref{C}.
The transposed matrix $A_{m.n}^{t}$
is Gale dual to the matrix $C_{n,m}$.
The matrix 
$C_{m,n}^{t}$ is $\Z$-Gale dual to the matrix $A_{n,m}$.
\end{proposition}

\begin{proof}
First, note that all matrices  $C_{n,m}$ and $A_{n,m}$ have maximal rank. Using Lemma ~\ref{L:bino} we obtain that $C_{n,m} \cdot A_{m,n}^{t}=0$
and thus $A_{m,n}^{t}$ is Gale dual to $C_{n,m}$ since $A_{m,n}^{t}$ has rank $m+1$. From $C_{n,m} \cdot A_{m,n}^{t}=0$ we get $A_{m,n} \cdot C_{n,m}^{t}=0$
and then $A_{n,m} \cdot C_{m,n}^{t}=0$ by permuting $m$ and $n$. It follows that $C_{m,n}^{t}$ is Gale dual to $A_{n,m}$ since $C_{m,n}^{t}$ has full rank $m$.
In fact $C_{m,n}^{t}$ is a $\Z$-Gale dual matrix of $A_{n,m}$ since it has one maximal minor which equals $\pm 1$.
\end{proof}

We first show that $\mb 1$ is an isolated solution of system~\ref{eq:fij}. For this, it is enough to prove that
our system has a finite number of solutions in the torus $(\C^*)^n$. Let $P_{n,m}$ be the convex-hull of 
$\A_{n,m}$. Note that $P_{n,m}$ is a cyclic polytope. In fact, using BKK Theorem
(see \cite{Be,Kho, Kush}), we show that
the sum of the multiplicities of the solutions in the complex torus of the system equals the lattice volume ${\rm vol}_\Z(P_{n,m})$ of
$P_{n,m}$, where ${\rm vol}_\Z(P)$ is equal to $n!$ times the Euclidean volume of $P$.

\begin{proposition}\label{prop:finite}
For any value of $n, m \ge 1$, the system  $f_{m+1, n}(z) = f_{m+2, n}(z) = \dots = f_{m+n, n}(z) = 0$ in~\eqref{eq:fij} has a finite
number of solutions in  $(\C^*)^n$.  In particular $\mb 1$ is an isolated common root.
\end{proposition}

\begin{proof}
It follows from~\cite{Sh} that $\A_{n,m}$ is the set of vertices of its convex-hull $P_{n,m}$ ($P_{n,m}$ is a cyclic polytope) and that
all proper faces of $P_{n,m}$ are simplices. Thus any $k$-dimensional face of $P_{n,m}$ with $k <n$ is a $k$-simplex intersecting
$\A_{n,m}$ at its $k+1$ vertices.

The coefficient matrix $C_{n,m}$ of system~\ref{eq:fij} is Gale dual to the transpose $A_{m,n}^t$ of the exponent matrix by Proposition~\ref{Gale duals}. 
It is straightforward to see that all maximal minors of $A_{m,n}$ are non-zero because  they are Vandermonde determinants. 
A basic result in Gale duality says that there is a non-zero constant $\lambda$
such that  any maximal minor of $C_{n,m}$ equals $\lambda$ time the complementary minor of $A_{m,n}^t$, and so we deduce that all maximal
minors of $C_{n,m}$ are non-zero. 
Write $\gamma(k) = (k, k^2,\dots, k^n)$ so that $\A_{n,m}=\{\gamma(0), \ldots, \gamma(m+n)\}$ and $f_{m+i,n}(x)= \sum_{k =0}^{m+i} (-1)^k\, {m+i \choose k} \, x^{\gamma(k)}$.

We now prove that for any face $F$  of $P_{n,m}$ of dimension $k$ with $1 \le k \le n-1$ the facial system
$f_{m+1,n}^{ F}=\cdots= f_{m+n,n}^{F}=0$ does not have any common root in the complex torus, where:
\begin{equation}\label{eq:fijk}
f_{m+i,n}^{F}(x) \, = \, \sum_{\{ k \le m+i \, : \, \gamma(k) \in F\}} (-1)^k\, {m+i \choose k} \, x^{\gamma(k)}.
\end{equation}
The coefficient matrix of system~\eqref{eq:fijk} is the submatrix $C^F$ of $C$ of size $n \times (k+1)$ obtained by removing the columns of $C$ 
corresponding to elements of $\A_{n,m}$ which are not in $F$.
Since all maximal minors of $C$ are non-zero, the matrix $C^F$ has full rank $k+1 \leq n$, hence there is a square submatrix of  $C^F$ of size $k+1$ which is invertible.
 This square submatrix cannot have a non-zero vector in its kernel, and it follows that system~\eqref{eq:fijk} cannot have a solution in the complex torus corresponding to
 the face $F$. 
 Then, there are finitely many solutions of system~\eqref{eq:fijk} in the complex torus $(\C^*)^n$   and
 by the BKK Theorem
(see \cite{Be,Kho, Kush}) the sum of their intersection multiplicities is equal to 
the lattice volume ${\rm vol}_\Z(P_{n,m})$. In particular $\mb 1$ is an isolated solution.
\end{proof}

We now state the main result of this section:

\begin{theorem}\label{conj}
The system $f_{m+1, n}(x) = f_{m+2, n}(x) = \dots = f_{m+n, n}(x) = 0$ has a root at 
$q=\mb 1$ of multiplicity at least ${n+m \choose n}$.
\end{theorem}

To prove Theorem~\ref{conj}, we set $f_k(x)=f_{m+k,n}(x)$ and $F_{k}(x)=f_{k}(x+\mb 1)$ for $k=1,\ldots,n$  as in~\eqref{eq:Fk}. We then compute 
the Newton diagram ${\mathcal D}(F_{k})$ of $F_{k}$.
For $\alpha \in \Z_{\geq 0}^{n}$ set
\begin{equation}\label{sum}
|\alpha|'=\sum_{j=1}^{n }j\alpha_{j}
\end{equation}
\begin{proposition}\label{main point}
Let $k=1,\ldots,n$. Then ${\mathcal D}(F_{k})$ is the union of the bounded faces of the convex-hull of the union of the sets $\alpha+\R_{\geq 0}^{n} $
over all $\alpha \in \Z_{\geq 0}^{n}$ such that $$|\alpha|'=\sum_{i=1}^{n }i\alpha_{i} \geq m+k.$$

Moreover, if $\alpha \in {\mathcal D}(F_{k}) \cap \Z_{\geq 0}^{n}$
then the coefficient of $x^{\alpha}$ in $F_{k}(x)$ is equal to
$\frac{1}{\alpha_1 ! \cdots \alpha_n !} \cdot c_{|\alpha|'}$, where $c_{|\alpha|'}$ is the coefficient of $z^{|\alpha|'}$ in the expansion as a series of
$|\alpha|' ! \cdot  (1-e^{z})^{m+k}$.
As a consequence ${\mathcal D}(F_{k})$ intersects the $i$-th coordinate axis at the point with $i$-th coordinate $\lceil \frac{m+k}{i}\rceil$ for $i=1,\ldots,n$ (where $\lceil \cdot \rceil$
stands for upper integer part).
\end{proposition}

\begin{proof}
Let $\alpha \in  {\ZZ}_{\geq 0}^{n}$. Then  $L_k(\alpha)=\sum_{i=j}^{n+m} c_{kj}a_{j}^{\alpha}$ equals
\[L_k(\alpha) = \sum_{j=1}^{n+m} c_{kj} \prod_{i=1}^{n}(j^{i})^{\alpha_{i}}=
\sum_{j=1}^{n+m} c_{kj} j^{|\alpha|'} =\sum_{j=1}^{m+n-k+1} {(-1)}^{j} {m+k \choose j}j^{|\alpha|'}.\]
Using Lemma ~\ref{L:bino}, we obtain that $L(\alpha)=0$ if $|\alpha|' \leq m+k-1$ and $L(\alpha) \neq 0$ if $|\alpha|' \geq m+k$. 
It remains to use Lemma ~\ref{lem:tor}.
\end{proof}

\medskip

\subsection*{Proof of Theorem~\ref{conj}}
\label{Proof of Theorem 1.2}
By Proposition~\ref{main point} the Newton polytope $\Delta_{k}$ of $F_{k}$ is convenient.
Hence, by Theorem~\ref{main co-mixed}, we have
$$\mult_{0}(F_1,\ldots,F_{n}) \geq \Vol^{\circ}(\Delta_{1},\ldots,\Delta_{n}).$$
Let $\sigma$ be the convex hull of the points $\frac{1}{i}e_{i}$ for $i=1,\ldots,n$, where $e_i$ is the $i$-th vector in the standard basis.
By Proposition~\ref{main point}, we have $\Delta_{k } +\R_{\geq 0}^{n} \subset Q_{k}+\R_{\geq 0}^{n}$, where
$Q_{k}=(m+k) \cdot \sigma$ for $k=1,\ldots,n$.
Therefore, by monotonicity of the mixed covolume (Proposition~\ref{decreasing property}), we obtain
$$\Vol^{\circ}(\Delta_{1},\ldots,\Delta_{n}) \geq \Vol^{\circ}(Q_{1},\ldots,Q_{n}).$$
We have $\Vol^{\circ}(Q_{1},\ldots,Q_{n})=(\prod_{k=1}^{n} m+k) \cdot \Vol^{\circ}(\sigma,\ldots,\sigma)$
by multilinearity of the mixed covolume. Moreover, we have $\Vol^{\circ}(\sigma,\ldots,\sigma)= n! \cdot \Vol(B_{\sigma})$. We compute that $\Vol(B_{\sigma})=(\frac{1}{n!})^{2}$
to conclude that $\Vol^{\circ}(Q_{1},\ldots,Q_{n})=(\prod_{k=1}^{n} m+k) \cdot \frac{1}{n!}={n+m \choose n}$.

\medskip

\begin{proposition} \label{prop:n=2}
Consider the system~\eqref{eq:fij}. For $n=2$ and any positive integer number $m$, the multiplicity at $\mb 1$ is equal to ${m+2 \choose 2}$.
\end{proposition}
\begin{proof}
Assume $n=2$. We treat the case $m$ odd as the case $m$ even is similar. Set $a=\frac{m+1}{2}$. We apply Proposition~\ref{main point}.
We compute that ${\mathcal D}(F_{1})$ is the segment $E_{1}=[(0,\frac{m+1}{2}),(m+1,0)]$ and the truncation of $F_{1}$ to $E_{1}$ is equal, up to a multiplicative constant, to the polynomial
$$\tilde{F}_{1}(x_{1},x_{2})=\sum_{i=0}^{a} c_{i} \cdot x_{1}^{2i} x_{2}^{a-i}, \quad \mbox{where} \quad  c_{i}= \frac{1}{(2i)! \cdot (a-i)!} .$$
Then, we compute that ${\mathcal D}(F_{2})$ is the union of the segments
\[E_{2,1}=[(0,\frac{m+3}{2}), (1,\frac{m+1}{2})]  \text{ and } E_{2,2}=[(1,\frac{m+1}{2}), (m+2,0)].\]
Moreover,
 the truncation of $F_{2}$ to $E_{2,2}$ is equal, up to a multiplicative constant, to 
$$\tilde{F}_{2}(x_{1},x_{2})=x_1 \cdot \sum_{i=0}^a c'_{i} \cdot x_{1}^{2i} x_{2}^{a-i}, \quad \mbox{where} \quad c'_{i}=\frac{1}{(2i+1)! \cdot (a-i)!}.$$
The truncation of $F_{2}$ to the other segment $E_{2,1}$ is a binomial polynomial (we will not need to know its coefficients).
Using Formula~\eqref{alternating sum}, one can check that $\Vol^{\circ}(\Delta_{1},\Delta_{2})={m+2 \choose 2}$, where $\Delta_{k}$ is the Newton polytope of $F_{k}$, see Figure~\ref{F:n=2}.
Thus, by Theorem~\ref{main co-mixed}, in order to prove that the multiplicity at the origin of the system $F_{1}=F_{2}=0$ is ${m+2 \choose 2}$, it suffices to prove that this system
is non-degenerate at the origin. For this it suffices to prove that the system $\tilde{F}_{1}=\tilde{F}_{2}=0$ has no solution in $(\C^{*})^{2}$.
Using the monomial change of coordinate $x=x_{1}^{2}x_{2}^{-1}$, we are reduced to show that the univariate polynomials
$\ell_{1}(x)=\sum_{i=0}^{a} c_{i}x^{i}$ and $\ell_{2}(x)=\sum_{i=0}^{a} c'_{i}x^{i}$ have no common complex root.
It turns out that, up to scalar multiplication and change of variable, these polynomial are known as Laguerre polynomials.
In fact, they are also Hermite polynomials, up to another change of variables and scalar multiplication.
Explicitely, $$\ell_{1}(x)=s_{1} \cdot L_{a}^{\frac{1}{2}}(-\frac{x}{4}) \quad \mbox{and} \quad
\ell_{2}(x)=s_{2} \cdot L_{a}^{-\frac{1}{2}}(-\frac{x}{4}),$$ where $s_{1}$ is the inverse of $\prod_{i=0}^{a-1}(a+\frac{1}{2}-i)$
and $s_{2}$  is the inverse of $\prod_{i=0}^{a-1}(a-\frac{1}{2}-i)$. It is known  by~\cite[Theorem 2.3]{DJ} that the Laguerre polynomials $L_{a}^{\frac{1}{2}}(x)$ and $L_{a}^{-\frac{1}{2}}(x)$
have only real roots (otherwise said, they are hyperbolic polynomials) and that their roots interlace. In particular, they have different roots. We conclude
that the multiplicity at the origin of the system $F_{1}=F_{2}=0$ is ${m+2 \choose 2}$.

\end{proof}

\begin{figure}[ht]
\begin{center}
\includegraphics[scale=0.5]{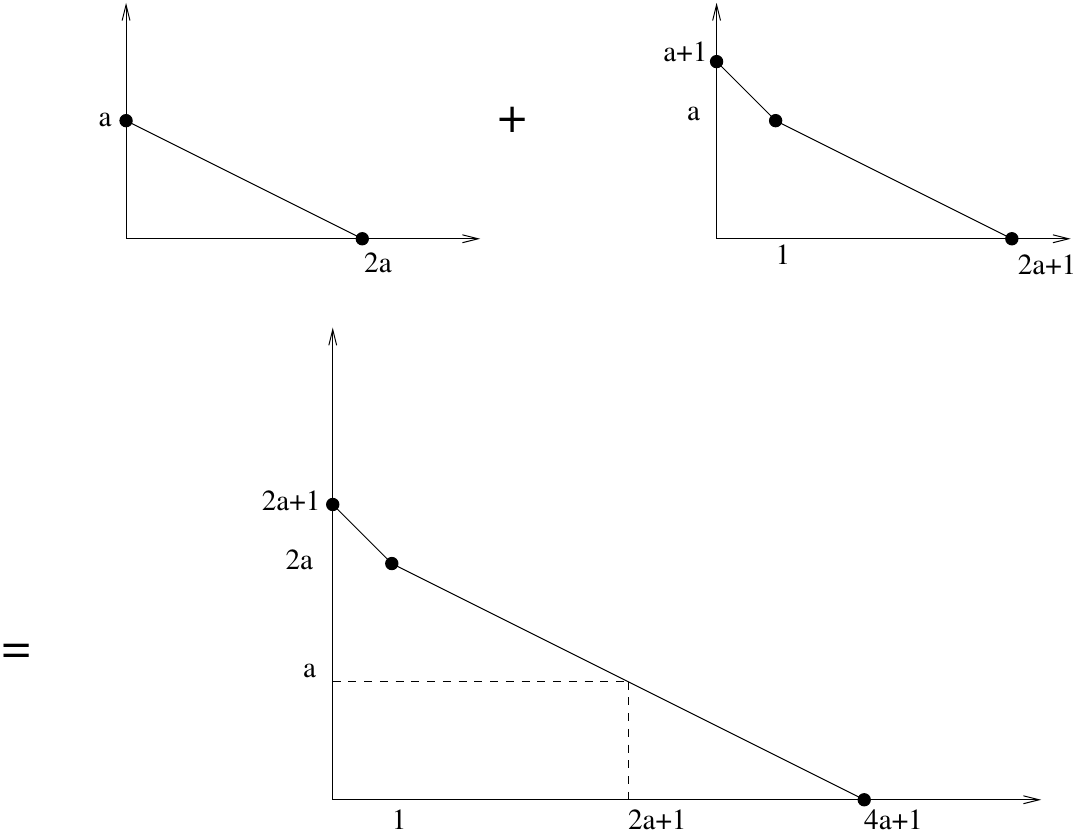}
\caption{The Newton diagrams ${\mathcal D}(F_{1}), {\mathcal D}(F_{2})$ and their Minkowski sum,
where $a= \frac{m+1}{2}$ ($m$ odd). We see that the mixed covolume equals $\Vol^{\circ}(\Delta_{1},\Delta_{2}) = a \cdot (2a+1)={m+2 \choose 2}$.}
\label{F:n=2}
\end{center}
\end{figure}

It is interesting to compare the  value of the maximum multiplicity at one point $m+2 \choose 2$ for $n=2$ and any codimension $m$ with the normalized volume of the cyclic polytope $P_{2,m}$ 
which is the convex hull of the corresponding support set, that bounds the number of isolated solutions.
Using the Online Encyclopedia of Integer Sequences (\url{https://oeis.org/}), we have that ${\rm vol_\Z} (P_{2,m}) = 2\, {m+3 \choose 3}$, 
a polynomial of degree $3$ in $m$. On the other side ${m+2 \choose 2}$ is only
a polynomial of degree $2$ in $m$, so the difference tends to infinity with $m$.

\begin{remark}
\label{rem:Gale duals}
We showed in Section~\ref{S:H} that $H$-duality associates to a polynomial system $C.(x+\mb 1)^A=0$ vanishing at $\mb 0 \in \C^n$ a polynomial system $B^{t}.(y+\mb 1)^{D^{t}}=0$ vanishing at $\mb 0 \in \C^m$, where
$B$ and $D$ are matrices (with $D$ reduced) which are Gale dual to $A$ and $C$ respectively. Proposition~\ref{Gale duals} reveals that
 $H$-duality just permutes the dimension and the codimension in the family of systems~\eqref{Jens} with high multiplicity (up to left equivalence for matrices).
\end{remark}

We end with a suggestion for future work.

\begin{remark}
We expect that the local multiplicity at $\mb 1$ of the system~\eqref{eq:fij} is equal to ${n+m \choose m}$ for any positive integers numbers $n,m$.
When $m=2$ a proof would follow closely the proof of Proposition~\ref{prop:n=2} using Theorem~\ref{th:dual}, Proposition~\ref{Gale duals}, see also Remark~\ref{rem:Gale duals}
(the Newton diagrams are the same, only the coefficients are slightly different). 
\end{remark}

\section{The codimension $1$ case} \label{sec:circuit}

In this section we give affirmative answers to Questions~\ref{construction} and~\ref{q:Mnm} in the case of codimension $1$ in Corollary~\ref{cor:M11} of Theorem~\ref{sharp bound for circuit} and Proposition~\ref{perturbation-circuit}. We also prove the interesting Theorem~\ref{pseudo Vandermonde}.

\begin{theorem}\label{sharp bound for circuit}
The maximal multiplicity at $\mb 1$  of a polynomial system with codimension $m=1$
and dimension $n$ is equal to $n+1$.
\end{theorem}
\begin{proof}
When $m=1$ the Gale system consists of one polynomial equation $g_{1}=0$ in one variable. Moreover, since $D$ has rank $2$, the polynomial $g_{1}$
 is not the zero polynomial. It follows that the Gale system is convenient and non-degenerate at $\mb 0$ and thus the multiplicity at $\mb 0$ is at most $n+1$ by
Corollary~\ref{C:t_{i}*}. By Theorem~\ref{Gale duality for one solution}, we get that the multiplicity of an isolated solution of a polynomial system with codimension $m=1$
and dimension $n$ is at most $n+1$. This bound is sharp by Theorem~\ref{conj}.
\end{proof}

\begin{corollary}\label{cor:M11}
$M(n,1) = M(1,n) = n+1$ for any $n \in \N$.
\end{corollary}
\begin{proof}
Theorem~\ref{sharp bound for circuit} says that $M(n,1) = n+1$. On the other side,
in the univariate case, any complex root of non-zero polynomial is isolated and the polynomial is convenient and non-degenerate at any complex root.  For any codimension $n$, the support $\A$ has cardinality $N=n+2$ and thus, as we mentioned at the beginning of the Introduction, the maximal multiplicity of a non-zero complex root equals $n+2-1=n+1$.
\end{proof}

We now characterize polynomial systems reaching the bound $n+1$ in Theorem~\ref{sharp bound for circuit}.
Let $\gamma: \C \rightarrow \C^{n}, t \mapsto \gamma(t)=(t,t^{2},\ldots,t^{n})$. A cyclic point configuration
is the image of a finite set by $\gamma$ and a cyclic polytope  is the convex hull of a cyclic point configuration in $\R^n$.
We refer to~\cite[Section 0]{Z} for the main combinatorial properties of these polytopes.

\begin{theorem}\label{pseudo Vandermonde}
Let $\A=\{a_{0},a_{1},\ldots,a_{n+1} \} \subset \R^{n}$ be a set of codimension $1$ and dimension $n$.
Let $C$ be a matrix of size $n \times (n+2)$ of maximal rank $n$. Then the system $C \cdot x^A=0$ has $\mb 1$ as a solution of multiplicity $n+1$
if and only if the kernel of $C$ contains the vector with all $1$'s coordinates and a vector with distinct 
coordinates $d_{0},\ldots,d_{n+1}$ such that $\A$ is equal to the cyclic point configuration $\{\gamma(d_{0}),\ldots,\gamma(d_{n+1})\}$ up to an affine transformation of $\R^{n}$.
\end{theorem}
\begin{proof}
Assume that the system $C \cdot x^A=0$ has $\mb 1$ as a solution of multiplicity $\mu=n+1$.
Then the kernel of $C$ contains the vector with all $1$'s coordinates and another vector with coordinates $(d_{0},d_{1},\ldots,d_{n+1})$ such that $d_{0}=0$.
This gives a reduced Gale dual matrix $D$ for $C$ with row vectors $(1,0), (1,d_{1}), \ldots,(1,d_{n+1})$.  Let $B$ be a Gale dual matrix of $A$.
Since $\A$ has codimension $1$, $B$ is a non-zero column matrix with coefficients $\lambda_{0},\lambda_{1},\ldots,\lambda_{n+1}$.
The associated Gale system is one non-zero polynomial equation in one variable $g_{1}=0$. Thus the Gale system is convenient and non-degenerate. Moreover it has 
 $\mb 0$ as a root of multiplicity  $n+1$ by Theorem~\ref{Gale duality for one solution}. Using Theorem~\ref{T:t_{i}*} we get that $d_{0},d_{1},\dots,d_{n+1}$ are distinct and $\lambda_{0},\lambda_{1},\ldots,\lambda_{n+1}$ are all non-zero. It follows then from Proposition~\ref{diagram} that 
\begin{equation}\label{caract}
\sum_{j=1}^{n+1} \lambda_{j}d_j^{i}=0, \; i=1,\ldots,n.
\end{equation}
Thus the matrix $A$ and the matrix $(d_{j}^i)_{0 \leq i \leq n, 0 \leq j \leq n+1}$ have the same Gale dual matrix $B$. Equivalently both matrices are equal up to left multiplication by an invertible matrix,
in other terms, $\A$ and the cyclic configuration $\{\gamma(d_{0}),\ldots,\gamma(d_{n+1})\}$ are equal up to affine transformation.

Conversely, assume that $\A$ is equal up to affine transformation to a
cyclic configuration $\{\gamma(d_{0}),\ldots,\gamma(d_{n+1})\}$, where all $d_{i}$ are distinct.
Then $\A$ is equal up to affine tranformation to the cylic configuration $\{\gamma(0),\gamma(d_1-d_{0}),\ldots,\gamma(d_{n+1}-d_{0})\}$ since both cyclic configurations share the same affine relations.
Then taking for $C$ the transpose of any Gale dual matrix of the matrix with a first row the all $1$'s vector and as a second row 
the vector $(0,d_1-d_{0},\ldots,d_{n+1}-d_{0})$ we get a system having $\mb 1$ as a solution of multiplicity $n+1$, by reversing the previous arguments.
\end{proof}
\begin{example}
\label{ex:square}
Assume that $\A=\{(0,0),(1,0),(1,1), (0,1)\}$. Then it is easy to show that $\A$ is not the image by an affine transformation of a cyclic point configuration $\{\gamma(d_{0}),\ldots,\gamma(d_{3})\}$, where all $d_{i}$ are distinct. Thus, the maximal multiplicity of a system $C \cdot x^{A}$ at a point 
$q \in (\C^{*})^{2}$ is at most $2$. In fact this maximal multiplicity is equal to $2$, see Example~\ref{ex:squarefirst}.
\end{example}

System ~\eqref{eqn:poly} has real coefficients, and it lead us to pose Question~\ref{construction}.
To our knowledge, a positive answer to this question would give a record number of positive solutions for any values of $m,n$ such that $m \geq 3$ and $n \geq 2$.
The answer is trivially yes for $n=1$.

When $m=1$ and $n$ is any positive integer, one tool used to get systems reaching the upper bound $n+1$ for the multiplicity is the theory of ''real dessins d'enfant''.
We are going to show below how to use real dessins d'enfant to affirmatively answer Question~\ref{construction} in the case $m=1$
(and any $n$). We recall some needed facts about real dessins d'enfant, a more complete description can be found in \cite{B} and the references therein.

Consider any polynomial map $\Phi =(P,Q) : {\C}P^1 \rightarrow {\C}P^1$,
where $P,Q$ are homogeneous polynomials of the same degree $d$. It is convenient to write $\Phi(x)=(\varphi(x):1)$ with $\varphi(x)=\frac{P(x)}{Q(x)}$ when $Q(x) \neq 0$.
Assume that $P$ and $Q$ have real coefficients. Then $\Phi$ is a real map, that is, we have $\overline{\Phi(x)}=\Phi(\bar{x})$ where the bar symbol denote the (coordinatewise) complex conjugation of $\C P^{1}$. 
The {\it real dessin d'enfant} associated to $\Phi$ is $\Gamma=\Phi^{-1}({\R}P^1)$. This is a graph on $\C P^{1}$ with vertices the critical points $x$ of $\Phi$ such that $\Phi(x) \in \R P^{1}$ (real critical values)
and edges
which form the pre-image of segments between consecutive two critical values in $\R P^{1}$. The inverse image of $(0:1)$ (resp., $(1:0)$, $(1:1)$) is the set of roots of $P$ (resp., $Q$, $g=P-Q$).
We will denote all roots of $P$ (resp., $Q$, $g=P-Q$) by the same letter $p$ (resp., $q$, $r$). The graph $\Gamma$ equipped with letters $p$, $q$, $r$ has the following properties.
It is invariant under complex conjugation. The restriction of $\Phi$ to any connected component of ${\C}P^1\setminus \Gamma$
is a covering of one connected component of $\C P^1 \setminus \R P^1$. The degree of this covering is given by the number of letters $p$ (resp., $q$, $r$) in the boundary and these letters
$p$, $q$, $r$ are arranged consistently with the arrangement of their images $(0:1)$ (resp., $(1:0)$, $(1:1)$ on $\R P^{1}$ (for instance, an open arc joigning letters $p$ and $q$ cannot contain the letter $r$ and
no other letters $p$, $q$). It follows that the degree of $\Phi$ is half the sum of all these local degrees. Moreover, the valency of a vertex of $\Gamma$ is an even integer number, in fact this is twice the multiplicity
of this vertex  as a critical point of $\Phi$. 

Conversely, any graph on $\C P^1$ marked with letters $p$, $q$, $r$ which satisfies the previous properties
is the real dessin d'enfant associated to a real polynomial map $\Phi: {\C}P^1 \rightarrow {\C}P^1$ by Riemann uniformization Theorem (see \cite{B} and the references therein for more details) whose degree
can be computed as described before.

Let $\A \subset \Z^{n}$ be a set of dimension $n$ codimension $m$. Consider a polynomial system $C \cdot X^{A}=0$ with support $\A$ and real coefficient matrix $C$.
The corresponding system~\eqref{rational} has the form $\prod_{i=0}p_{i}(y)^{\lambda_i}=1$ where all $p_{i}$ are real degree one polynomials given by the row of a matrix $D$ which is Gale dual to $C$
and the exposants $\lambda_{i}$ are integer coefficients in any non trivial affine relation on $\A$.
So the solutions of the Gale system are the pre-image of $1$ under this rational function, that is, the marked points $r$ on the corresponding real dessin d'enfant.
Note that all marked points $p$ and $q$ belong to $\R P^{1}$ since the numerator and denominator of the rational functions are product of real polynomials of degree one.
Conversely, any map $\Phi$ with this property can be seen as a rational map coming, via Gale duality, from a real polynomial system of dimension $n$ and codimension $1$. 
Given a real linear transformation $\tau$ of the source space $\C P^{1}$, one can consider the image of $\Gamma$ by $\tau$ and 
mark by the same letter $p,q$ or $r$ the image of any marked point of $\Gamma$. Choosing another Gale dual matrix $D$ for $C$ 
is the same as applying such a transformation, so we see a real dessin d'enfant up to this group action to  get rid off the dependence 
on the choice of $D$.

\begin{proposition}
\label{perturbation-circuit} Let $\A \subset \Z^{n}$ be a set of dimension $n$ and codimension $1$, and let $C \in \R^{n \times (n+2)}$ be a matrix
such that the corresponding polynomial system $C \cdot X^{A}=0$ has a solution at $\mb 1$ of multiplicity $n+1$. Then there exists a small perturbation
$\tilde{C}  \in \R^{n \times (n+2)}$ of $C$ such that the polynomial system 
$\tilde{C} \cdot X^{A}=0$ has $n+1$ positive solutions.
\end{proposition}
\begin{proof}
Choose any non trivial affine relation with integer coefficients $\lambda_{i}$ on $\A$.
Choose any reduced Gale matrix $D$ of $C$ and consider the rational function $\varphi=\varphi_{1}$ given by~\eqref{rational} (here $m=1$).
Recall that $p_i(y)=1+d_{i} y$ for $i=0,\ldots,n+1$, where the $(1,d_{i})$ are the rows of $D$. Then, by Theorem~\ref{Gale duality for one solution}
we have that $y=0$ is a critical point of multiplicity $n+1$ of $\varphi$ with critical value $1$, which corresponds to the 
solution $(1,\ldots,1)$ of the system  $C \cdot X^{A}=0$. Let $\Gamma$ be the associated real dessin d'enfant. The critical point $y=0$ 
is a marked point of $\Gamma$, with a letter $r$, and it is a vertex of $\Gamma$ with valency $2(n+1)$. By Gale duality Theorem for positive solutions
 (see \cite{BS}), there is a bijection (the restriction of the map $\psi$ defined in the proof of Theorem~\ref{Gale duality for one solution})
between the set of positive solutions of the polynomial system $C \cdot X^{A}=0$ and the set of marked points with letters $r$ contained 
in the interval $I =\{y \in \R \, , \, p_i(y)>0 \; \mbox{for all $i$}\}$. Since $y=0$ is the image of $x=(1,\ldots,1)$ by the Gale map, we see that $I$
is the open interval in $\R P^1$ containing the marked point $y=0$ and no marked points with letters $p$ and $q$. We might deform slightly 
$\Gamma$ in a neighbourhood $U$ of the point $y=0$ in order to get a real dessin d'enfant $\Gamma$ with $n+1$ marked points with letter 
$r$ in the interval $I$ (and $n$ additional vertices of valency $2$ on $\R P^{1}$ in $I$, one between two consecutive letters $r$). This is the
 local deformation of real dessins d'enfant called \emph{splitting} in \cite{B} (see Figure~\ref{Splitting} which corresponds to the case $n=2$).
 Then, the resulting real dessin d'enfant has $n+1$ 
 marked point with letter $r$ in $I$, and the marked points with letters $p$ and $q$ still belong to $\R P^{1}$ and have the same valencies. 
 It follows that the resulting real dessin d'enfant corresponds to a real polynomial system with same support (since the affine relation is preserved)
 having the maximal number $n+1$ of positive solutions and which is a small perturbation of the starting system. 
\end{proof}

\begin{figure}[ht]
\begin{center}
\includegraphics[scale=0.6]{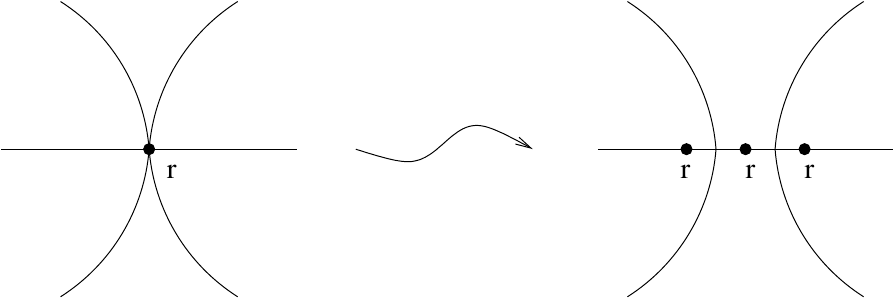}
\caption{A local deformation of a real dessin d'enfant with a triple point (valency six) marked with a letter $r$ to a real dessin d'enfant with three simple points (valency two each) marked with a letter $r$. The horizontal segment represents part ot $\R P^{1}$.}
\label{Splitting}
  \end{center}
\end{figure}

\section{Multiplicity of one point on one hypersurface} \label{sec:onef}
%

In this section we consider as before a configuration $\A  \subset \Z^n$ of affine dimension $n$ and
codimension $m$, and a single non-zero  polynomial $f $ with support in $\A$. As we remarked, the condition that $\A$ has affine dimension $n$ 
is precisely the condition that the matrix $A \in \Z^{(n+1)\times N}$ defined in~\eqref{eq:A} has maximal rank $n+1$. Moreover, all maximal minors 
of the matrix $A$ are non-zero if and only if any subset of $n+1$ elements in the configuration $\A$ is affinely independent. 
In this case,  the matrix $A$ is  said to be {\em uniform}. Note that this is a generic condition in the space of matrices.

As noted before, it is enough to consider  the case  when $\A \subset \Z^n_{\ge 0}$. Then,  $f$ lies in $\C[x_1, \dots, x_n]$
 and the multiplicity of $f$ at any point is at most its degree ${\rm deg}(f)$. 
 We give in this section two bounds on the possible multiplicity $\mu(f;\mb q)$ of $f$ at a point $\mb q \in (\C^*)^n$ only in terms of the dimension $n$ and the codimension $m$.  
 
 We start with a {\em naive}  bound $\sigma(n,m)$ in Proposition~\ref{proposition:naive} that holds whenever the points in $\A$ do not lie in a hypersurface of small degree. 
 In Theorem~\ref{thm:MultiplicityOfOnePoint} we present a new bound $b(n,m)$ for uniform configurations. 
 We show in Proposition~\ref{prop:sharp} that $b(n,m)$ is always attained for cyclic configurations. 
 We also show in Theorem~\ref{th:2} that for uniform configurations $\A$ of dimension $2$ and codimension $m$,  
 $\sigma(2,m) \le b(2,m)$ and we prove that  $b(2,m)$ is a sharp bound for {\em any} planar configuration with affine dimension $2$. We end with a  counterexample in dimension $3$.
 
 Recall that the number of lattice points in the $d$-th dilate of the unit simplex in $\R^n$, that is, the number
 of monomials in $n$ variables of degree up to $d$, equals the combinatorial number $n+d \choose n$. 
 Let  $\A = \{ a_0, \dots, a_{m+n}\} \subset \Z_{\ge 0}^n$ be a configuration of affine dimension $n$ and
codimension $m$, and consider $f = \sum_{i=0}^{n+m} c_j \, x^{a_j}$ with support in $\A$.
By definition, the multiplicity $\mu(f;\mb q)$ of $f$ at $\mb q$ is at least  $\mu$ if
 all derivatives $\partial^\alpha(f)(\mb q)=0$ for any $\alpha \in \Z_{\ge 0}^n$ with $|\alpha| \le \mu-1$ (and it equals $\mu$ 
 if moreover some partial derivative of order $\mu$ does not vanish at $\mb q$).  
 Thus, $\mu(f;\mb 1) \ge \mu$ if and only if the vector of coefficients $c=(c_0, \dots, c_{m+n})$ lies in the kernel of the matrix 
 $A^{(\mu-1)} \in \Z^{{n+\mu-1 \choose n }\times (n+m+1)}$ in Definition~\ref{def:Ak}.
 Clearly, this matrix can have full row rank and non-trivial kernel only when ${n+\mu-1 \choose n} \le (n+m)$.

\begin{lemma} \label{lem:p}
Fix $n$ and let $p$ be the univariate polynomial $p(\mu) = {n+ \mu-1 \choose n}- (n+m)$. Then $p$ has a unique positive root $\mu_0(n,m)$. 
\end{lemma}

\begin{proof} As $p(0) < 0$ and its leading coefficient is positive, $p$ has at least one positive root. It has exactly one because of 
Descartes' rule of signs since the sign variation of its coefficients is equal to $1$. \end{proof}

Note that for any $k \in \N$,  $p(k) \le 0$ when $k \le \mu_0(n,m)$ and $p(k) > 0$ only when $k > \mu_0(n,m)$. 

 \begin{definition}\label{def:sigmab}
Let $n,m$ natural numbers. 
The associated naive bound $\sigma(n,m)$ is defined as the floor $\lfloor{\mu_0(n,m)}\rfloor$, that
is, as the highest (positive) integer smaller or equal than $\mu_0$.
 \end{definition}

We next show when the multiplicity $\mu$ is bounded by $\sigma(n,m)$. 

\begin{proposition}
\label{proposition:naive}
Let $\A = \{ a_0, \dots, a_{n+m}\} \subset \Z^n_{\ge 0}$ be a configuration of affine dimension $n$ and codimension $m$. 
\begin{itemize}
\item[(i)] The matrix $A^{(\mu-1)}$  in Definition~\ref{def:Ak} has full row rank if and only if there is no polynomial of degree at most $\mu-1$ 
that vanishes over $\A$.
\item[(ii)] Assume that $A^{(\mu-1)}$ has full row rank. Let $f = \sum_{i=0}^{n+m} c_j x^{a_j}$ be a non-zero
polynomial such that
 $\mb 1$ is an isolated root with multiplicity $\mu$ of $f$. Then, $\mu \le \sigma(n,m)$.
 \end{itemize}
\end{proposition}

\begin{proof}
As pointed out in the proof of Lemma~2.5 in \cite{DP}, it is easy to see that for any non-zero vector 
$\lambda \in \C^{n+ \mu -1 \choose n}$, $\lambda$ lies in the {\em left} kernel of $A^{(\mu-1)}$ if and only
if the non-zero polynomial of degree up to $\mu-1$ with vector of coefficients $\lambda$ vanishes over $A$. 
This proves item (i).

Assume now that $A^{(\mu-1)}$ has full row rank and  $\mb 1$ is an isolated root with multiplicity $\mu$ of the non-zero
polynomial $f = \sum_{i=0}^{n+m} c_j x^{a_j}$. Then, the non-zero vector $c$ of coefficients of $f$
 lies in the (right) kernel of the matrix $A^{(\mu-1)}$ but not in the kernel of the matrix $A^{(\mu)}$.
As we pointed out in the paragraph before Lemma~\ref{lem:p}, if $A^{(\mu -1)}$ has full row rank and non-trivial kernel, then the polynomial 
$p$ defined in Lemma~\ref{lem:p} verifies that $p(\mu) \le 0$ and we then deduce that
$\mu \le \sigma(n,m)$.
\end{proof}
 
We now introduce the numbers $b(n,m)$.

\begin{definition}\label{def:b}
Given natural numbers $m,n$,  we call $b(n,m)$ the number defined as $1 + \lceil{\frac m n}\rceil$, that is, as the smallest (positive) integer greater
or equal than $1 + \frac m n$.    
\end{definition}
 
\begin{theorem}
\label{thm:MultiplicityOfOnePoint}
Consider  a  uniform configuration  $\A = \{ a_0, \dots, a_{m+n}\} \subset \Z^n$  of cardinality $N$ and codimension $m = N-n-1 \ge 0$. 
Then, 
for any non-zero polynomial $f (x)= \sum_{i=0}^{m+n} c_i \, x^{a_i}$ with support $\A$ and any point $\mb q \in V(f) \subset (\CC^*)^n$,  we have that
\begin{equation}\label{eq:mu2}
\mu(f; \mb q)\,  \le \,  b(n,m) \, = \, 1+\left\lceil m/n \right\rceil \,  = \, \lceil (N-1)/n \rceil.
\end{equation}
\end{theorem}

\begin{proof}
Let $\A\subset \Z^n$ be a finite uniform configuration, $f$ a non-zero polynomial with support $\A$ as in the statement and $\mb q \in (\C^*)^n$ such that $f(\mb q)=0$.
Our proof is by induction on the codimension $m$ of $\A$.

The induction basis consists of the case $m=0$
for arbitrary $n$. Indeed, if $m = 0$ then $\A$ consists of $n+1$ affinely independent points and the associated matrix 
$A\in \Z^{(n+1)\times(n+1)}$ is invertible. If $\mu(f; \mb q) >1 $ (that is, if all partial derivatives of $f$ vanish at
$\mb q$),  we get that the non-zero vector $(c_0 \mb q^{a_0}, \dots, c_{N-1} \mb q^{a_{N-1}})$ lies in the kernel of the matrix $A$, a contradiction. 
Then,  each point $\mb q\in V(f)$ is regular and $\mu(f; \mb q) \le 1 \le  1 +\left\lceil m/n \right\rceil$.

Assume now that $m\ge 1$ and 
let  $H = \{ \alpha \in \R^n \, : \, \langle v_H, \alpha \rangle = \gamma_H\}$ be a hyperplane containing $n$ points in $\A$. 
Note that $H$ cannot contain more points in $\A$ by our hypothesis that $A$ is uniform. Let $a \in \A \cap H$. Without loss of generality, we can consider the translated configuration
$\A - a$, which corresponds to dividing all monomials in a Laurent polynomial $f$ with support in $\A$ by the monomial $x^{a}$. 
Taking $v_H \in \Z^n$  primitive (i.e. the gcd of its coordinates equals $1$), we 
can also assume, up to performing an invertible monomial change of coordinates in the torus $(\C^*)^n$, that
$v_H = (0,\dots, 0,1)$ and $H = (\alpha_n=0)$.
Then any $f$ with support in $\A$ can be written as:
\[
 f  \, = f(x_1, \dots, x_n) \, = \,  f_1(x_1, \dots, x_{n-1}) \, + \, f_2(x_1, \dots, x_n),  
\]
where all monomials in $f_2$ lie in ${\A}'=\A \setminus H$.
It follows that for any $\mb q \in (\C^*)^n$, 
\begin{equation} \label{eq:efe2}
\mu(f; \mb q) \,  \le \,  1+\mu(\theta_n(f); \mb q) \, = \,
1+\mu(\theta_n (f_{2}); \mb q),
\end{equation}
where we recall that $\theta_n(f) = x_n \frac{\partial f}{\partial x_n}$.
Note that $\theta_n (f_{2})$ has also support in ${\A}'$.

In case $m < n$, the cardinality of ${\A}'$ is smaller than $n$ and it consists of affinely independent points since 
otherwise any maximal minor of the matrix $A$ containing the columns corresponding to ${\A}'$ would be zero. 
Repeating our previous argument, we see that $\mu(\theta_n(f_2); \mb q) =0$ and thus
$\mu(f; \mb q) \le 1 + 0 \le 1 +\left\lceil m/n \right\rceil =2$.

If instead $m \ge n$, ${\A}'$ is a uniform configuration of affine dimension $n$  with codimension $0 \le m-n < m$. 
It follows from ~\eqref{eq:efe2} and our inductive hypothesis that
\[\mu(f; \mb q) \, \le \, 2+ \lceil (m-n)/n\rceil \, = \, 1 +\lceil m/n\rceil,\]
as claimed.
\end{proof}

If $n = 1$, then a hyperplane in $\RR$ is a point. Hence, $\A$ fulfills the genericity assumption
of Theorem~\ref{thm:MultiplicityOfOnePoint}. In particular, we have a proof of the result
we mentioned in the Introduction: If $\A$ is a point configuration with dimension $n=1$ and codimension $m$, 
then the multiplicity of a point $\mb q \in V(f) \subset \CC^*$ is at most $1+m$.

\begin{proposition}\label{prop:sharp}
The bound from Theorem~\ref{thm:MultiplicityOfOnePoint} is sharp for  any $n,m$.
\end{proposition}

\begin{proof}
The polynomial $f_{n+m,n}$ from~\eqref{eqn:poly} is given by
\[
f_{n+m,n}(x) = \sum_{k=0}^{n+m}(-1)^k\, {n+m \choose k} \,x_1^k \,x_2^{k^2} \cdots x_n^{k^n}.
\]
Notice that for by Lemma~\ref{L:bino}, for each $j < n+m$ it holds that
\[
\sum_{k=0}^{n+m}(-1)^k\, k^j {n+m \choose k} \,x_1^k \,x_2^{k^2} \cdots x_n^{k^n}\Bigg|_{x=(1, \dots,1)} = 0.
\]
It follows that the multiplicity of $f_{n+m,n}$ at $(1, \dots,1)$ is equal to the smallest $s$
such that $s\cdot n \geq n+m$. This finishes the proof.
\end{proof}

The assumptions in Theorem~\ref{thm:MultiplicityOfOnePoint} are not optimal. That is,
there are plenty of non-uniform point configurations $\A$ for which one can find a proof using
the two cases of the induction step. This is the case of
configurations of dimension $2$, as we show in Theorem~\ref{th:2} below. On the other hand, the 
bound $b(n,m)$ does not hold in general (see Example~\ref{ex:HighMultiplicity}).

Note that for any fixed $n$, ${n + \mu-1} \choose n$ is a polynomial in $\mu$ of degree $n$, while $b(n,m)$ is linear in $m$.
It follows that after a value of $m$, we have that the naive bound $\sigma(n,m)$ is strictly smaller than our bound $b(n,m)$. 
Moreover, we have shown in Proposition~\ref{prop:sharp}
that the bound $b(n,m)$ is attained.

\begin{theorem}\label{th:2}
If $\A$ is a point configuration with dimension $n=2$ and codimension $m$, 
then the multiplicity of a point $\mb q \in V(f) \subset \CC^*$ is at most $b(2,m)$.
Moreover we always have that $\sigma(2,m) \le b(2, m)$ and the  inequality is strict for any $m \ge 5$.
\end{theorem}

\begin{proof}
While a point configuration $\A$ of dimension $2$ need not fulfill the genericity assumption
of Theorem~\ref{thm:MultiplicityOfOnePoint}, we show that one can always find a hyperplane $H$
as in one of the two cases in the induction step in the proof of the theorem.

If all edges of the Newton polytope $\Delta(\A)$ contain at least three points, 
then we can choose $H$ as the affine span of one edge of $\Delta(\A)$; the complement $\A\setminus H$
is full-dimensional since $\Delta(\A)$ has at least three edges.

Assume that there is an edge of $\Delta(\A)$ which contains exactly two points.
Let $H$ be the affine span of this edge. 
There are two possibilities. Either $\A\setminus H$ is full-dimensional, or $\A \setminus H$
is contained in a line  $L$. In the latter case, $\A\setminus L$ consists
of the vertices of a simplex.
The proof of the last statement is straightforward and we only illustrate it in Table~\ref{demo-table}.
\end{proof}

\begin{table}[!h] \label{demo-table}
\begin{center}
\begin{tabular}{||c c c||}
 \hline
m & $\sigma(2,m)$ & $b(2,m)$  \\ [0.5ex] 
 \hline\hline
 1 & 2 & 2 \\ 
 \hline
 2 & 2 & 2 \\
 \hline
 3 & 3 & 3 \\
 \hline
 4 & 3 & 3 \\
 \hline
 5 & 3 & 4\\ 
 \hline
  6 & 3 & 4 \\ 
 \hline
 7 &3 & 5 \\
 \hline
 8 & 4 & 5 \\
 \hline
 9 & 4 & 6\\
 \hline
 10 & 4 & 6 \\ 
  [1ex] 
 \hline
\end{tabular}
\caption{Comparing bounds for $n=2$.}
\end{center}
\end{table} 

We end with an example showing that the  bound from Theorem~\ref{thm:MultiplicityOfOnePoint} does not necessarily hold for non-generic 
uniform configurations in dimension $n=3$.
\begin{example}
\label{ex:HighMultiplicity}
A simple example in dimension $3$ of a non-uniform configuration $\A$ for which the bound
in Theorem~\ref{thm:MultiplicityOfOnePoint} does not hold  is given by the polynomial
\[
f(x_1, x_2, x_3) = x_1 g(x_2) + g(x_3), \quad \text{where } g(z) = (1-z)^4.
\]
The support set of $f$ has codimension $m = 10-4 =6$; the multiplicity at $(1,1,1)$
is $4$, but $1 + \lceil 6/3\rceil = 3$. Note that if we consider the support $\A = \{(1,0,0), (1,1,0), (1,2,0),
(1,3,0), (1,4,0), (0,0,0), (0,0,1), (0,0,2), (0,0,3), (0,0,4)\}$ the corresponding discriminantal variety is not a hypersurface.
\end{example}


\end{document}